\documentclass[a4paper,11pt]{article}

\usepackage{xcolor}
\usepackage[pdfencoding=unicode, hidelinks]{hyperref}
\hypersetup{
	colorlinks,
	linkcolor={red},
	citecolor={blue},
}
\usepackage[english]{babel}
\usepackage[utf8]{inputenc}
\usepackage{amsthm}
\usepackage{amsmath}
\usepackage{amssymb}
\usepackage{mathrsfs}
\usepackage[mathscr]{eucal}
\usepackage{wasysym}
\usepackage{textcomp}
\usepackage{enumerate}
\usepackage{setspace}
\usepackage{latexsym}
\usepackage{graphicx}
\usepackage{mathtools}
\usepackage{tabularray}
\usepackage{geometry}
\usepackage{float}
\usepackage{emptypage}
\usepackage{commath}
\usepackage{multicol}
\usepackage{titlesec}
\usepackage[shortlabels]{enumitem}
\usepackage{verbatim}
\usepackage{cleveref}

\usepackage{csquotes}
\usepackage[backend=biber,
    style=alphabetic,
    citestyle=alphabetic,
    doi=false,
    url=false,
    isbn=false,
    maxnames=999,
    abbreviate=true,
    giveninits=true]{biblatex}
\usepackage{hyperref}
\AtEveryBibitem{%
  \clearfield{note}%
}

\theoremstyle{plain}
\newtheorem{thm}{Theorem}[section]
\theoremstyle{plain}

\theoremstyle{plain}
\newtheorem{lemma}[thm]{Lemma}
\theoremstyle{plain}
\newtheorem{defn}[thm]{Definition}
\theoremstyle{plain}
\newtheorem{prop}[thm]{Proposition}
\theoremstyle{definition}
\newtheorem{remark}[thm]{Remark}
\numberwithin{equation}{section}

\newcommand{\RR}{\mathbb{R}}

\newcommand{\NN}{\mathbb{N}}
\newcommand{\dd}{\,\ensuremath{\mathrm{d}}}
\newcommand{\dx}{\,\ensuremath{\mathrm{d}}x\,}
\newcommand{\ds}{\,\ensuremath{\mathrm{d}}s\,}
\newcommand{\dt}{\,\ensuremath{\mathrm{d}}t\,}
\newcommand{\dxt}{\,\ensuremath{\mathrm{d}}x \, \ensuremath{\mathrm{d}}t}
\newcommand{\dtau}{\,\ensuremath{\mathrm{d}}\tau\,}
\newcommand{\dS}{\,\ensuremath{\mathrm{d}}S\,}
\renewcommand{\norm}[1]{\lVert#1\rVert}
\newcommand{\duality}[2]{\langle#1,#2\rangle}

\newcommand{\sigmag}{\sigma_{\Gamma}}

\newcommand{\A}{\mathcal{A}}
\newcommand{\B}{\mathcal{B}}
\newcommand{\uu}{\boldsymbol{u}}
\newcommand{\uut}{\partial_t\boldsymbol{u}}
\newcommand{\epsilonu}{\varepsilon(\boldsymbol{u})}
\newcommand{\epsilonut}{\varepsilon(\partial_t\uu)}
\newcommand{\vv}{\boldsymbol{v}}

\newcommand{\oomega}{\boldsymbol{\omega}}
\newcommand{\epsilonv}{\varepsilon(\boldsymbol{v})}
\newcommand{\epsilonvt}{\varepsilon(\partial_t\vv)}

\newcommand{\epsilonul}{\varepsilon(\boldsymbol{u}_l)}
\newcommand{\epsilonult}{\varepsilon(\partial_t\boldsymbol{u}_l)}

\newcommand{\ff}{\boldsymbol{f}}
\newcommand{\nnu}{\boldsymbol{\nu}}

\newcommand{\Ws}{W}
\newcommand{\Wsz}{W_0}

\newcommand{\Vs}{V}
\newcommand{\Hs}{H}

\DeclareMathOperator{\diver}{div}
\newcommand{\into}{\int_{\Omega}}
\newcommand{\intTo}{\int_0^T \!\!\!\into}

\newcommand{\intto}{\int_0^t \!\!\!\into}
\newcommand{\intt}{\int_0^t}
\newcommand{\inttau}{\int_0^{\tau}}
\newcommand{\inttauo}{\int_0^{\tau}\!\!\!\into}
\newcommand{\ints}{\int_0^s}
\newcommand{\intg}{\int_{\Gamma}}

\newcommand{\intk}{\int_{t_{\tau}^{k-1}}^{t_{\tau}^k}}

\newcommand{\fk}{\ff_{\tau}^k}
\newcommand{\phik}{\varphi_{\tau}^k}
\newcommand{\Ak}{\mathcal{A}_{\tau}^k}
\newcommand{\Bk}{\mathcal{B}_{\tau}^k}
\newcommand{\uk}{\uu_{\tau}^k}
\newcommand{\epsilonuk}{\varepsilon(\uu_{\tau}^k)}
\newcommand{\vk}{\vv_{\tau}^k}

\newcommand{\zk}{z_{\tau}^k}

\newcommand{\vj}{\vv_{\tau}^j}

\newcommand{\ukm}{\uu_{\tau}^{k-1}}
\newcommand{\epsilonukm}{\varepsilon(\uu_{\tau}^{k-1})}

\newcommand{\phit}{\varphi_{\tau}}
\newcommand{\ut}{\uu_{\tau}}
\newcommand{\uti}{\widehat{\uu}_{\tau}}
\newcommand{\vt}{\vv_{\tau}}
\newcommand{\zt}{z_{\tau}}
\newcommand{\ft}{\ff_{\tau}}
\newcommand{\At}{\A_{\tau}}
\newcommand{\Bt}{\B_{\tau}}

\crefname{lemma}{lemma}{lemmas}
\Crefname{lemma}{Lemma}{Lemmas}
\crefname{prop}{proposition}{proposition}
\Crefname{prop}{Proposition}{Propositions}
\crefname{cor}{corollary}{corollaries}
\Crefname{cor}{Corollary}{Corollaries}
\crefname{remark}{remark}{remarks}
\Crefname{remark}{Remark}{Remarks}
\crefname{thm}{theorem}{theorems}
\Crefname{thm}{Theorem}{Theorems}
\Crefname{section}{Section}{Sections}

\begin{document}

\begin{center}
		
{\Large \bf On a Brain Tumor Growth Model \\with Lactate Metabolism, Viscoelastic Effects, \\ and Tissue Damage}

		\vskip0.5cm
		
		{\large\textsc{Giulia Cavalleri$^1$}} \\
		{\normalsize e-mail: \texttt{giulia.cavalleri01@universitadipavia.it}} \\
		\vskip0.35cm

		{\large\textsc{Pierluigi Colli$^2$}} \\
		{\normalsize e-mail: \texttt{pierluigi.colli@unipv.it}} \\
		\vskip0.35cm

        {\large\textsc{Alain Miranville$^3$}} \\
		{\normalsize e-mail: \texttt{alain.miranville@univ-lehavre.fr}} \\
		\vskip0.35cm
		
		{\large\textsc{Elisabetta Rocca$^2$}} \\
		{\normalsize e-mail: \texttt{elisabetta.rocca@unipv.it}} \\
		\vskip0.35cm
		
		{\footnotesize $^1$Department of Mathematics ``F. Casorati'', University of Pavia, 27100 Pavia, Italy}
		\vskip0.1cm
		
        {\footnotesize$^2$Department of Mathematics ``F. Casorati'', University of Pavia \& IMATI-C.N.R., 27100 Pavia, Italy}
        \vskip0.1cm

         {\footnotesize$^3$ School of Mathematics and Statistics, Henan Normal University, Xinxiang, P. R. China \& Université Le Havre Normandie, Laboratoire de Mathématiques Appliquées du Havre (LMAH), 25, rue Philippe Lebon, BP 1123, 76063 Le Havre cedex, France}
        \vskip0.5cm
		
	\end{center}

	\begin{abstract}\noindent
In this paper, we study a nonlinearly coupled initial-boundary value problem describing the evolution of brain tumor growth including lactate metabolism. In our modeling approach, we also take into account the viscoelastic properties of the tissues as well as the reversible damage effects that could occur, possibly caused by surgery. After introducing the PDE system, coupling a Fischer--Kolmogorov type equation for the tumor phase with a reaction-diffusion equation for the lactate, a quasi-static momentum balance with nonlinear elasticity and viscosity matrices, and a nonlinear differential inclusion for the damage,  we prove the existence of global in time weak solutions under reasonable assumptions on the involved functions and data. Strengthening these assumptions, we subsequently prove further regularity properties of the solutions as well as their continuous dependence with respect to the data, entailing the well-posedness of the Cauchy problem associated with the nonlinear PDE system. 
        
		\vskip3mm
		

		\vskip3mm
		
		
	\end{abstract}

\date{}

\section{Introduction}

In this paper, we study an initial-boundary value problem describing the dynamics of a brain tumor including lactate metabolism, viscoelastic effects of the tissues as well as their possible damage. 
Four nonlinearly coupled  PDEs describe the evolution of the concentrations of tumor denoted by $\varphi$, the intracellular lactate
$\sigma$, the ``small'' displacement $\uu$, and the damage parameter $z$.  The equation for $\sigma$ is based on the derivation done in \cite{Aubert_etal_2005} and in \cite{Guillevin_etal_2011}. Actually, in the first brain lactate kinetics models in the literature (cf.~\cite{Cherfils_Gatti_Guillevin_Miranville_Guillevin22}, \cite{Guillevin_etal_2018} and references therein) the authors dealt with the evolution of both capillary and intracellular lactate concentrations. However, since 
here we had in mind to include the displacement and damage evolution in the model,  we neglect the capillary lactate concentration in order to simplify the PDE system, which, in our case, turns out to be the following one:
\begin{subequations}\label{eq:problem}
    \begin{align}
        & \partial_t\varphi - \Delta \varphi = p(\sigma,z)\varphi\left(1-\frac{\varphi}{N}\right) - \varphi g(\sigma,z), \label{eq:phi}\\
        & \partial_t\sigma - \Delta \sigma + \frac{k_1(\varphi,z)\sigma}{k_2(\varphi,z) + \sigma} = J(\varphi,z), \label{eq:lactate}\\
        &  - \diver\left[\A(\varphi,z) \epsilonut + \B(\varphi,z)\epsilonu  \right] = \ff, \label{eq:displacement}\\
        & \partial_t z - \Delta z + \beta(z) + \pi(z)  \ni  w - \Psi(\varphi,\epsilonu), \label{eq:damage}
    \end{align}
\end{subequations}
posed in $Q \coloneqq \Omega \times (0,T)$, where $\Omega$ is a bounded $C^2$ domain in $\RR^n$ with $n=2,3$ and $T>0$ is a fixed time. Regarding the boundary conditions, we assume the system is isolated from the exterior, so we prescribe no-flux conditions for $\varphi$ and $z$. Regarding the lactate $\sigma$, we allow more general
Robin conditions, with the physical constants set to 1 for simplicity. Since our model is specifically designed for brain tumors, the domain is confined by a rigid boundary, the cranium, which prevents displacement at the boundary. Consequently, it is a natural choice to impose homogeneous Dirichlet boundary conditions for the displacement $\uu$. Therefore, we
couple the previous system with the following boundary conditions:
\begin{subequations}\label{eq:boundary_conditions}
    \begin{align}
        & \partial_{\nnu}\varphi = \partial_{\nnu}z = 0,\label{eq:phi,z_neumann}\\
        & \partial_{\nnu}\sigma =\sigma_{\Gamma} -\sigma,\label{eq:sigma_robin}\\
        & \uu = \mathbf{0}\,. \label{eq:u_dirichlet}
    \end{align}
\end{subequations}
Finally, we consider the following usual initial conditions: 
\begin{equation}\label{eq:initial_conditions}
    \varphi(0) = \varphi_0,\quad
    \sigma(0) = \sigma_0, \quad
    \uu(0) = \uu_0, \quad
    z(0) = z_0\,. 
\end{equation}
The coupling between equations \eqref{eq:phi} and \eqref{eq:lactate} was studied in \cite{Cherfils_Gatti_Guillevin_Miranville_Guillevin22}, while models including the effects of the stress (reducing the proliferation of the tumor) were already introduced and studied in \cite{Garcke_Lam_Signori_21}, where a Cahn--Hilliard type dynamics for the tumor concentration was used in place of the Fischer--Kolmogorov type equation \eqref{eq:phi}. 
Finally, let us mention the recent work \cite{cavalleri_2024}, by one of the authors of the present contribution, where a system coupling tumor growth dynamics of Cahn--Hilliard type together with displacement and damage was first analyzed and the existence of weak solutions was proved for the corresponding initial-boundary value problem. 
In our setting, the Fischer--Kolmogorov type equation \eqref{eq:phi} (cf.~\cite{fisher_1937}, \cite{Kolmogorov_1937}) describes the evolution of tumor cell concentration, denoted by $\varphi$, which takes values in the interval $[0,N]$. The constant $N$ represents the carrying capacity, which is defined as the maximum number of cells per unit volume. The balance between proliferation and
apoptosis is taken into account through the rate $p$, while we account for necrosis through the rate $g$, where both $p$ and $g$ possibly depend on the lactate proportion $\sigma$ and the damage parameter $z$.  

The reaction-diffusion equation \eqref{eq:lactate} represents the evolution of intracellular lactate production $\sigma$. Since lactate diffusion from cells to capillaries is coupled with the transport of hydrogen ions $H+$, in symport
terms we have positive functions $k_1$ and $k_2$, representing the ratios between a Michaelis–Menten
constant and the concentration of $H+$  in the intracellular space and in the capillary
compartment, respectively. 
The positive function $J$ collects the production of lactate in cells by glycolysis, the consumption of lactate by metabolism, and the diffusion of lactate in neighboring regions. 
Moreover, we assume that $k_1$, $k_2$, and $J$ are possibly dependent on the other variables $\varphi$ and $z$. 

The main novelty of this contribution relies on the fact that here we include the mechanics in the model by assuming a viscoelastic behavior of biological tissues. 
It is indeed well known that solid stress can aﬀect tumor growth (see e.g.~\cite{Urcun_Lorenzo_etal_22}) and, at the same time, tumor growth increases mechanical stress. Due to the complexity of the problem, as a first approach here we
assume infinitesimal displacements, so we work in the case of linear elasticity (cf.~also \cite{Garcke_Lam_Signori_21} and \cite{Garcke_Lam_Signori_Opt_Contr_2021} for similar derivations). 
The evolution of the small displacement $\uu$ is then ruled by the vectorial quasi-stationary balance law \eqref{eq:displacement}, where the two tensors $\mathcal{A}$ and $\mathcal{B}$ describe the elastic and viscous effects. They may depend on the tumor and damage variables, as well as on $\sigma$ in a non-degenerating way (cf.~assumptions \eqref{hyp:AB} later on). For a more detailed discussion of this dependency, we refer to the following \Cref{remark:sigma_dep}. Finally, $\ff$ represents a given volume force. 
Let us notice that the study of the mechanical effects in tumor growth models results particularly important in the case of brain (glioma) tumors, where the evolution strongly depends on the tissues (see e.g., \cite{Alfonso_etal_2017} and references therein). 

Finally, following the ideas developed in \cite{cavalleri_2024},  we consider, through equation \eqref{eq:damage}, the possible damage effects that could occur for example in case of surgery that causes lesions which, in turn, affect
the proliferation of tumor cells.  Following the derivation in Fr\'emond and Nedjar (cf.~\cite{Fremond_Nedjar_95}, \cite{Fremond_Nedjar_96}, \cite{Fremond_2002}), the evolution of the damage parameter $z\in [0,1]$ ($z=0$ means the tissue is completely damaged and $z=1$ means completely safe) is ruled by the evolution inclusion \eqref{eq:damage}, where the maximal monotone graph $\beta$ having bounded domain (in $[0,1]$) forces the variable $z$ to assume the physically meaningful values in between $0$ and $1$. A simple choice for $\beta$ could be the subdifferential of the indicator function of the set $[0,1]$, which takes value $0$ in $[0,1]$ and is $+\infty$ elsewhere. The function $\pi$ denotes a regular, possibly non-monotone function, $w$ represents an energy threshold for initiation of damage, while $\Psi$ describes the coupling between the damage and the displacement along with the tumor concentration. Usually, in damage models (cf., e.g., \cite{Heinemann_Kraus_2011}, \cite{Rocca_Rossi_2014}, \cite{Heinemann_Rocca_2015}, \cite{Mielke_Roubicek_Zeman_2010}, \cite{Thomas_Mielke_2010}) the dependence of $\Psi$ on $\epsilonu$ is quadratic because it comes from the derivative of the elastic part of the energy with respect to the damage variable $z$. However, here we cannot handle a quadratic dependence and so we assume a Lipschitz-continuity dependence of $\Psi$ with respect to both $\varphi$ and $\epsilonu$ (cf.~assumption \ref{hyp:psi} later on), cf.~also \cite{Kuttler_Shillor_2006}.

The main mathematical difficulties arise here from here on the nonlinear coupling between equations and, in particular, on the dependence of the elasticity and viscosity matrices in \eqref{eq:displacement} by $\varphi$ and $z$. 

The technique used in the proof relies on a suitable maximum principle coupled with Moser estimates to prove the boundedness of $\varphi$ and $\sigma$, as well as a fixed-point argument, suitable a priori estimates and compactness results entailing on the sufficient weak, weak-star and strong convergences in the approximating sequences in order to pass to the limit and identify the nonlinearities.

Hence, the plan of the paper is the following: in the next Section~\ref{sec:main} we set the assumptions on the data and we state our main results concerning the global (in time) existence of weak solutions, the regularity results under more restrictive assumptions on the data, as well as the continuous dependence of solutions with respect to the data. In Subsection~\ref{sec:th1} we prove the first result through a Schauder fixed-point argument including a maximum principle and a Moser argument in order to prove the constraints on the variables $\varphi$ and $\sigma$ as well as a time discretization of the displacement equation \eqref{eq:displacement} and a regularization of the graph $\beta$ by means of its Yosida approximation $\beta_\lambda$. The proof is then concluded by passing to the limit in both the time step $\tau$ and the regularization parameter $\lambda$. 
In Section~\ref{sec:th2} we prove a higher regularity result which is then needed in order to get uniqueness and continuous dependence of solutions with respect to the data (cf.~Section~\ref{sec:th3}). 

Many relevant issues arise from this problem and will be the subject of further investigations, like the optimal control problem that is of particular importance when administrating drugs or antiangiogenic therapies to the patient. In this direction, we refer to \cite{Cherfils_Gatti_Guillevin_Miranville_Guillevin22} and \cite{Guillevin_etal_2018}, where damage and mechanical effects are not taken into account. For a different perspective, we also refer to \cite{Garcke_Lam_Signori_Opt_Contr_2021}, where the proposed system, unlike ours, is not specifically tailored for brain tumors and again, it does not include tissue damage.

\section{Weak Formulation and Statement of Main Results}
\label{sec:main}
\subsection{Notation and preliminaries}
\textbf{Notation.} In what follows, for any real Banach space $X$ with dual space $X'$, we indicate its norm as $\norm{\cdot}_{X}$ and the dual pairing between $X'$ and $X$ as $\duality{\cdot}{\cdot}_{X}$. We denote the Lebesgue and Sobolev spaces over $\Omega$ as $L^p \coloneqq L^p(\Omega)$, $W^{k,p} \coloneqq W^{k,p}(\Omega)$ and $H^k \coloneqq W^{k,2}(\Omega)$, while for the Lebesgue spaces over $\Gamma$ we use $L^p_{\Gamma} \coloneqq L^p(\Gamma)$. 
For convenience, we set $\Hs \coloneqq L^2(\Omega)$ and we identify $\Hs$ with its dual space $\Hs'$. We introduce $\Vs \coloneqq H^1(\Omega)$ and $\Vs_0\coloneqq H^1_0(\Omega)$, where $H^1_0(\Omega)$ represents the set of $H^1(\Omega)$ functions with zero trace at the boundary. Additionally, we define $\Ws$, the space of $H^2(\Omega)$ functions with zero normal derivative at the boundary, and $\Wsz \coloneq H^2(\Omega) \cap \Vs_0$, representing $H^2(\Omega)$ functions with zero trace at the boundary. In both cases, the natural norm 
induced by $H^2$ is denoted by $\norm{\cdot}_{\Ws}$.
To simplify the notation, we do not always distinguish between scalar, vector, and matrix-valued spaces.  For instance, we use $L^p$ to indicate both $L^p(\Omega)$ and $L^p(\Omega;\RR^n)$, depending on the context. However, we adopt bold font to denote vectors and calligraphic font to denote tensors. 
For the sake of brevity, the norm of the Bochner space $W^{k,p}(0,T;X)$ is indicated as $\norm{\cdot}_{W^{k,p}(X)}$, omitting the time interval $(0,T)$. Sometimes, for $p \in [1,+\infty)$, we could identify $L^p(Q)$ with $L^p(0,T;L^p)$. With the notation $C^0([0,T];X)$ we mean the space of continuous $X$-valued functions. 
Finally, as is customary, we use $C$ to represent a generic constant depending only on the problem's data and whose value might change from line to line. If we want to highlight a dependency on a certain parameter, we put it as a subscript (e.g., $C_{\tau}$ indicates a constant that depends on $\tau$, $C_0$ a constant that depends on the initial data, etc.). \\

\noindent \textbf{Useful inequalities.} We will make use of classical inequalities, such as H\"older, Young, Gronwall, Poincaré, Poincaré--Wirtinger, Ehrling (see \cite[][Theorem 16.4, p. 102]{Lions_Magenes_12}), and Gagliardo--Nirenberg (see e.g. \cite{Nirenberg_59}).
In particular, we will employ the following special cases of  Gagliardo--Nirenberg inequality:

\begin{lemma}\label{lemma:special_case_gagliardo_nirenberg}
    Let $\Omega \subseteq \RR^n$ be a bounded Lipschitz domain. Then, it exists a constant $C$ such as, for every $v \in \Vs$, it holds
    \begin{align}
         \norm{v}_{L^4} \leq C \norm{v}^{\frac{1}{2}}_{\Hs}\norm{v}^{\frac{1}{2}}_{\Vs} &\qquad \text{ if } n=2,\\
        \norm{v}_{L^3} \leq  C \norm{v}^{\frac{1}{2}}_{\Hs}\norm{v}^{\frac{1}{2}}_{\Vs} &\qquad\text{ if } n=3.
    \end{align}
\end{lemma}

\noindent \textbf{Mathematical visco-elasticity.} Let $\A=(a_{hijk})$ be a fourth-order tensor over $\RR^2$. We say that:
    \begin{enumerate}[(i)]
        \item $\A$ is \textit{symmetric} if
        \begin{equation}\label{cond_sym}
            a_{hijk}(\varphi,z) = a_{ihjk}(\varphi,z) = a_{jkhi}(\varphi,z)
        \end{equation}
        for a.e. $(\varphi,z) \in \RR^2$ and for every indices $i,j,k,h = 1, \dots, d$.
        \item $\A$ is \textit{strictly positive definite} (or \textit{uniformly elliptic}) if there exists a positive constant $C$ such that for all $\epsilon \in \RR^{n\times n}_{\text{sym}}$ and for a.e. $(\varphi,z)\in \RR^2$
        \begin{align} \label{cond_ellipt}
        \A(\varphi,z) \epsilon : \epsilon \geq C |\epsilon|^2,
        \end{align} 
        where $:$ denotes the standard Frobenius inner product between matrices.
    \end{enumerate}

\noindent The following regularity result can be applied.

\begin{lemma} \label{lemma:H2_inequality}
    Let $\Omega$ be a $C^2$ domain in $\RR^n$ and  $\A=(a_{ijkh}) \in W^{1,\infty}(\Omega;\RR^{n \times n \times n \times n})$ be a symmetric and strictly positive definite fourth-order tensor.
   Then, there exist $C_*, C^* > 0$ such that for every $\uu$ in $\Wsz$ 
    \begin{equation*}
        C_* \norm{\uu}_{\Ws} \leq \norm{\diver[\A\, \epsilonu]}_{\Hs} \leq C^* \norm{\uu}_{\Ws}.
    \end{equation*}
\end{lemma}
\noindent For more details, cf. \cite[][Proposition 1.5, p. 318]{marsden1994} and \cite[][Lemma 3.2., p. 263]{necas2012}.\\

\noindent As will be specified later, throughout this work, the viscous tensor $\A$ will enjoy the properties \eqref{cond_sym}, \eqref{cond_ellipt}. Regarding the elastic tensor $\B$, it will be symmetric and positive definite, in the following sense.
\begin{enumerate}[(i),resume]
    \item $\B$ is \textit{positive definite}, if,  for all $\epsilon \in \RR^{n \times n}_{\text{sym}}$ and for a.e. $(\varphi,z)\in \RR^2$,
    \begin{align} \label{cond_weak_ellipt}
    \B (\varphi,z) \epsilon : \epsilon \geq 0.
    \end{align} 
\end{enumerate}

\subsection{Hypotheses}
In the following, we enlist the hypotheses we will work with throughout the paper. 

\begin{enumerate}[(\rm{A\arabic*})]
\item \label{hyp:eq_phi}
We suppose that
\begin{align}
    &p, \, g : \RR^2 \to \RR \text{ are continuous functions,}\\
    &0 \leq p \leq p^*, \quad 0 \leq g \leq g^*,
\end{align}
where $p^*$, $g^*$ are positive constants, and that
\begin{equation}
   N \text{ is a positive constant.}
\end{equation}

\item \label{hyp:eq_lactate}
We assume that
\begin{align}
    &k_1, k_2, J : \RR^2 \to \RR \text{ are  continuous and}\\
    &0 \leq k_1 \leq k_1^*, \quad 0 < {k_2}_* \leq k_2 \leq k_2^*, \quad  0 \leq J \leq J^*, 
\end{align}
where $k_1^*$, ${k_2}_*$, $k_2^*$, $J^*$ denote fixed positive constants.

\item \label{hyp:eq_displacement}
We require that $\A = (a_{ijkh}), \B = (b_{ijkh}) : \RR^2 \to \RR^{n \times n \times n \times n}$ are fourth-order tensors such that 
\begin{align}\label{hyp:AB}
    & \A, \, \B \text{ are of class } C^1 \text{ and bounded along with their partial derivatives,}\\
    &\A \text{ is strictly positive definite,}\\
    & \B \text{ is positive definite}.
\end{align}
Moreover, we assume
\begin{equation*}
\ff \in L^{\infty}(H).
\end{equation*}

\item \label{hyp:beta}
We suppose that there exists a $\widehat{\beta}: \RR \to [0,+\infty]$ such that 
\begin{equation}
    \mathcal{D}(\widehat{\beta}) \subseteq [0,1], \quad \widehat{\beta} \text{ is proper, l.s.c. and convex} 
\end{equation}
and we denote its subdifferential by $\beta \coloneqq \partial \widehat{\beta} : \RR \rightrightarrows \RR$.

\item \label{hyp:pi}
We consider a function $\widehat{\pi} \in C^1(\RR)$ and we denote by $\pi \coloneqq \widehat{\pi}'$ its derivative, requiring that
\begin{align}
    &\widehat{\pi} \text{  is concave,}\\
    &\pi \text{ is Lipschitz continuous.}
\end{align}

\item \label{hyp:w}
We suppose that 
\begin{equation}
    w \in L^{\infty}(0,T;\Hs).
\end{equation}

\item \label{hyp:psi}
We assume that $\Psi : \Omega \times \RR \times\RR^{n \times n} \to \RR$ is 
\begin{align}
    &\text{a Carathéodory function.}
\end{align}
Moreover, we require that 
\begin{equation}\label{hyp:psi_lipschitz}
    \begin{split}
        &\Psi(x, \cdot, \cdot) :  \RR \times \RR^{n \times n} \to \RR \text{ is Lipschitz continuous }, \text{i.e.,}\\
    & \exists C_{\Psi}>0 \text{ s.t. } |\Psi(x,\varphi_1,\epsilon_1)-\Psi(x,\varphi_2,\epsilon_2)| \leq C_{\Psi} \left(|\varphi_1-\varphi_2| + |\epsilon_1 - \epsilon_2|\right) 
    \end{split}
\end{equation}
for a.e. $x \in \Omega$, for all $\epsilon_1, \epsilon_2 \in \RR^{n \times n}$, $\varphi_1, \varphi_2 \in \RR$ and that 
\begin{equation}\label{hyp_psi_0}
    \hat{\Psi}(x) \coloneqq \Psi(x,0,\mathbf{0}) \in \Hs.
\end{equation}
\end{enumerate}

\begin{remark}
    Notice that from Hypotheses \eqref{hyp:psi_lipschitz} and \eqref{hyp_psi_0} we  trivially deduce that
    \begin{equation}\label{eq:ineq_psi}
          |\Psi(x,\varphi,\epsilon)| \leq  |\Psi(x,\varphi,\epsilon)-\Psi(x,0,\mathbf{0})| + |\Psi(x,0,\mathbf{0})|\leq C_{\Psi} \left(|\varphi| + |\epsilon|\right)  + |\hat{\Psi}(x)|
    \end{equation}
    for a.e. $x \in \Omega$ and for all $\varphi \in \RR, \epsilon \in \RR^{n \times n}$.
\end{remark}

\noindent In the following, for the sake of brevity, we will omit the dependence of $\Psi$ on the point $x$ in the notation, using $\Psi(\varphi,\epsilon)$ instead of $\Psi(x,\varphi,\epsilon)$.

\begin{enumerate}[(\rm{A\arabic*}),resume]
\item \label{hyp:boundary_conditions}
Regarding the boundary conditions, we suppose that
\begin{align}
    \sigma_{\Gamma} \in L^{2}(L^2_{\Gamma}),\quad 0 \leq \sigma_{\Gamma} \leq M_0,
\end{align}
where $M_0$ is a fixed positive constant.

\item \label{hyp:initial_conditions}
Regarding the initial conditions, we require that
\begin{align}
    &\varphi_0 \in \Vs, \quad 0 \leq \varphi_0 \leq N,\\
    &\sigma_0 \in \Hs, \quad 0 \leq \sigma_0 \leq M_0, \label{eq:hp_sigma0}\\
    &\uu_0 \in \Vs_0,\\
    &z_0 \in \Vs, \quad \widehat{\beta}(z_0) \in L^1.
\end{align}
\end{enumerate}

\noindent To prove additional regularity and continuous dependence, we will need the following stronger hypotheses. We assume that:

\begin{enumerate}[(\rm{B\arabic*})]
\item\label{hyp:uniqueness_eq_phi}
the functions $p, g$ are of class $C^1$ and are bounded along with their partial derivatives,

\item \label{hyp:uniqueness_eq_lactate}
 $k_1, k_2, J$ are Lipschitz continuous,

\item \label{hyp:uniqueness_w}
$w \in H^1(0,T;\Hs)$,

\item\label{hyp:uniqueness_initial_conditions}
$\varphi_0 \in \Ws$, $\uu_0 \in \Wsz$, $z_0 \in \Ws$ and $\beta^0(z_0) \in \Hs$, where $\beta^0$ is the minimal section of $\beta$.
\end{enumerate}

\section{Main results}
\begin{defn}\label{defn:weak_solution}
    We say that the quadruplet $(\varphi,\, \sigma,\, \uu,\, z)$ is a weak solution to the PDE system \eqref{eq:problem} endowed with the boundary and initial conditions \eqref{eq:boundary_conditions}--\eqref{eq:initial_conditions} if  
    \begin{gather*}
            \varphi \in   H^1(0,T;\Hs) \cap L^{\infty}(0,T; \Vs) \cap L^{2}(0,T;\Ws), \quad 0 \leq \varphi \leq N,\\
            \sigma \in H^1(0,T;\Vs') \cap L^{\infty}(0,T;\Hs) \cap L^2(0,T;\Vs), \quad 0 \leq \sigma \leq M,\\
            \uu \in W^{1,\infty}(0,T;\Vs_0)\\
            z \in   H^1(0,T;\Hs) \cap L^{\infty}(0,T; \Vs) \cap L^{2}(0,T;\Ws),
    \end{gather*}
   where $M=M(M_0, J^*)$, with 
   \begin{gather*}
       \varphi(0)=\varphi_0, \quad \sigma(0)=\sigma_0, \quad  \uu(0)=\uu_0,  \quad z(0)=z_0
   \end{gather*}
   a.e. in $\Omega$ and there exists a subgradient
   \begin{equation*}
       \xi \in L^2(0,T;\Hs), \quad \xi \in \beta(z) \text{ a.e. in } Q
   \end{equation*}
   such that
    \begin{subequations}\label{eq:problem_eq_with_spaces}
        \begin{align}
        & \into \partial_t\varphi \zeta +  \nabla \varphi \cdot \nabla \zeta \dx = \into \bigg[p(\sigma,z)\varphi\left(1-\frac{\varphi}{N}\right) - \varphi g(\sigma,z)\bigg] \zeta \dx, \label{eq:phi_with_spaces}\\
        &  \duality{\partial_t\sigma}{\zeta }_{\Vs} + \into \nabla \sigma \cdot \nabla \zeta + \frac{k_1(\varphi,z)\sigma \zeta }{k_2(\varphi,z) + \sigma}\dx + \intg (\sigma - \sigmag) \zeta \dS = \into J(\varphi,z) \zeta \dx,\label{eq:sigma_with_spaces}\\
        &  \into  \left[\A(\varphi,z) \epsilonut + \B(\varphi,z)\epsilonu  \right] : \varepsilon(\vv) \dx = \into \ff \cdot \vv \dx,\label{eq:u_with_spaces}\\
        & \into \partial_t z \zeta + \nabla z \cdot \nabla \zeta  + \xi \zeta  + \pi(z) \zeta \dx = \into \left[w -  \Psi(\varphi,\epsilonu)\right] \zeta \dx,\label{eq:z_with_spaces}
    \end{align}
\end{subequations}
a.e. in $(0,T)$, for every $\zeta \in \Vs$ and $\vv \in \Vs_0$.
\end{defn}

\begin{remark}
    Notice that, with the regularity we require, a solution of \eqref{eq:problem}--\eqref{eq:initial_conditions}  in the sense of \Cref{defn:weak_solution} satisfies equation \eqref{eq:phi} and inclusion \eqref{eq:damage} a.e. in $Q$.
\end{remark}
\begin{thm}[Existence] \label{thm:existence}
    Under the set of Hypotheses \textrm{(A)}, the PDE system \eqref{eq:problem} endowed with the boundary and initial conditions \eqref{eq:boundary_conditions}--\eqref{eq:initial_conditions} admits at least one weak solution in the sense of \Cref{defn:weak_solution}. 
\end{thm}

\begin{remark} \label{remark:sigma_dep}
    As mentioned in the Introduction,  it is possible to allow the tensors $\A$, $\B$ to depend on the lactate $\sigma$. However, this would require $\sigma$ to have the same regularity as $\varphi$ and $z$ (see Step IV of the existence proof). To ensure this, we would need the additional assumptions $\sigma_0 \in \Vs$ and $\sigmag \in H^1(0,T;L^2_{\Gamma}) \cap L^2(0,T;H^{1/2}_{\Gamma})$. Nevertheless, we do not point out this additional regularity for $\sigma$ in the statement of the following regularity theorem, also because it is not needed for the continuous dependence result stated in \Cref{thm:continous_dependence}.
\end{remark}

\begin{thm}[Regularity]\label{thm:regularity}
    Under the set of Hypotheses \textrm{(A)} and \textrm{(B)}, the solution to the PDE system \eqref{eq:problem}--\eqref{eq:initial_conditions} we found in \Cref{thm:existence} enjoys the further regularity
    \begin{gather*}
        \varphi \in H^1(0,T;\Vs) \cap L^{\infty}(0,T;\Ws) \cap L^2(0,T;H^3),\\
         \uu \in W^{1,\infty}(0,T;\Wsz),\\
         z \in H^1(0,T;\Vs) \cap L^{\infty}(0,T;\Ws),
    \end{gather*}
    and the subgradient satisfies
    \begin{equation*}
        \xi \in L^{\infty}(0,T;\Hs).
    \end{equation*}
\end{thm}
    
 \begin{thm}[Continuous dependence and uniqueness]  \label{thm:continous_dependence}
    Under the set of Hypotheses \textrm{(A)} and \textrm{(B)}, for every pair $\{(\varphi_i,\sigma_i,\uu_i,z_i)\}_{i=1,2}$ of weak solutions to \eqref{eq:problem}--\eqref{eq:initial_conditions} corresponding to the initial data $\{(\varphi_{0,i},\sigma_{0,i},\uu_{0,i},z_{0,i})\}_{i=1,2}$ and to the assigned functions $\{( \ff_i, w_i, \sigma_{\Gamma,i})\}_{i=1,2}$, the following continuous dependence inequality holds
    \begin{equation*}
        \begin{split}
            \norm{\varphi_1-&\varphi_2}_{L^{\infty}(\Hs)\cap L^2(\Vs)} + \norm{\sigma_1-\sigma_2}_{L^{\infty}(\Hs)\cap L^2(\Vs)}
            + \norm{\uu_1-\uu_2}_ {H^1(\Vs_0)} \\
            &+  \norm{z_1-z_2}_{L^{\infty}(\Hs)\cap L^2(\Vs)}\\
            &\begin{split}
                \leq C \Big(\norm{\varphi_{0,1} - \varphi_{0,2}}_{\Hs}&+\norm{\sigma_{0,1}-\sigma_{0,2}}_{\Hs}+\norm{\uu_{0,1}-\uu_{0,2}}_{\Hs} +\norm{z_{0,1}-z_{0,2}}_{\Hs}\\
                &+ \norm{\ff_1-\ff_2}_{L^2(\Hs)} + \norm{w_1 - w_2}_{L^2(\Hs)} +  \norm{\sigma_{\Gamma,1}-\sigma_{\Gamma,2}}_{L^2(L^2_{\Gamma})}  \Big)
             \end{split}
        \end{split}
    \end{equation*}
    for a positive constant $C$ that only depends on the problem data. In particular, the solution of \eqref{eq:problem} coupled with the boundary conditions \eqref{eq:boundary_conditions} and the initial conditions \eqref{eq:initial_conditions} is unique.
\end{thm}

\subsection{Proof of \Cref{thm:existence}}
\label{sec:th1}

The proof of \Cref{thm:existence} will be performed through a Schauder fixed-point argument. To do it properly, the first step consists of proving the existence of weak solutions for an approximate system. 

\subsubsection{The approximate system}
The approximate problem is obtained by replacing the maximal monotone operator $\beta$ with its Yosida approximation $\beta_{\lambda} \coloneqq ({\widehat{\beta}_{\lambda}})'$, where $\lambda \in (0,\lambda^*)$ is intended to go to $0$ in the limit. Moreover, we introduce the truncation function
\begin{equation}\label{eq:alpha}
    \alpha(\varphi) \coloneqq \begin{cases}
        \displaystyle \varphi \left(1 - \frac{\varphi}{N} \right) &\text{ if } 0 \leq \varphi \leq N\\
        0 &\text{ otherwise}\\
    \end{cases}
\end{equation}
and we use it in the \eqref{eq:phi} in order to have a bounded term on the right-hand side instead of a quadratic one. Finally, in equation \eqref{eq:lactate} we substitute the denominator $k_2(\varphi,z) + \sigma$ with $k_2(\varphi,z) + |\sigma|$ to be sure that it does not vanish. This way, we get the approximate PDE system:
\begin{subequations}\label{eq:problem_lambda}
    \begin{align}
        & \partial_t\varphi - \Delta \varphi = p(\sigma,z)\alpha(\varphi) - \varphi g(\sigma,z), \label{eq:phi_lambda}\\
        & \partial_t\sigma - \Delta \sigma + \frac{k_1(\varphi,z)\sigma}{k_2(\varphi,z) + |\sigma|} = J(\varphi,z), \label{eq:lactate_lambda}\\
        &  - \diver\left[\A(\varphi,z) \epsilonut + \B(\varphi,z)\epsilonu  \right] = \ff, \label{eq:displacement_lambda}\\
        & \partial_t z - \Delta z + \beta_{\lambda}(z) + \pi(z)  =  w - \Psi(\varphi,\epsilonu) .\label{eq:damage_lambda}
    \end{align}
\end{subequations}

\begin{defn}\label{defn:weak_solution_approximate}
    We say that the quadruplet $(\varphi,\, \sigma,\, \uu,\, z)$ is a weak solution to the approximate PDE system \eqref{eq:problem_lambda} endowed with the boundary and initial conditions \eqref{eq:boundary_conditions}--\eqref{eq:initial_conditions} if  
    \begin{gather*}
            \varphi \in   H^1(0,T;\Hs) \cap L^{\infty}(0,T; \Vs) \cap L^{2}(0,T;\Ws),\\
            \sigma \in H^1(0,T;\Vs') \cap L^{\infty}(0,T;\Hs) \cap L^2(0,T;\Vs), \\
            \uu \in W^{1,\infty}(0,T;\Vs_0)\\
            z \in   H^1(0,T;\Hs) \cap L^{\infty}(0,T; \Vs) \cap L^{2}(0,T;\Ws),
    \end{gather*}
   with 
   \begin{gather*}
       \varphi(0)=\varphi_0, \quad \sigma(0)=\sigma_0, \quad  \uu(0)=\uu_0,  \quad z(0)=z_0\quad \text{a.e. in }\Omega
   \end{gather*}
  and it satisfies 
    \begin{subequations}\label{eq:problem_eq_with_spaces_approx}
        \begin{align}
        & \into \partial_t\varphi \zeta +  \nabla \varphi \cdot \nabla \zeta \dx = \into \left[p(\sigma,z)\alpha(\varphi) - \varphi g(\sigma,z)\right] \zeta \dx,\label{eq:phi_with_spaces_approx}\\
        & \into \partial_t\sigma \zeta + \nabla \sigma \cdot \nabla \zeta + \frac{k_1(\varphi,z)\sigma \zeta }{k_2(\varphi,z) + |\sigma|}\dx + \intg (\sigma - \sigmag) \zeta \dS = \into J(\varphi,z) \zeta \dx,\label{eq:sigma_with_spaces_approx}\\
        &  \into  \left[\A(\varphi,z) \epsilonut + \B(\varphi,z)\epsilonu  \right] : \varepsilon(\vv) \dx = \into \ff \cdot \vv \dx,\label{eq:u_with_spaces_approx}\\
        & \into \partial_t z \zeta + \nabla z \cdot \nabla \zeta  + \beta_{\lambda}(z) \zeta  + \pi(z) \zeta \dx = \into \left[w -  \Psi(\varphi,\epsilonu)\right] \zeta \dx,\label{eq:z_with_spaces_approx}
    \end{align}
\end{subequations}
a.e. in $(0,T)$, for every $\zeta \in \Vs$ and $\vv \in \Vs_0$.
\end{defn}

\noindent Notice that, if $(\varphi, \sigma, \uu, z)$ is a solution of \eqref{eq:problem_lambda}, \eqref{eq:boundary_conditions}, \eqref{eq:initial_conditions} in the sense of \Cref{defn:weak_solution_approximate} and we are able to prove that $0 \leq \varphi \leq N$ and $0 \leq \sigma \leq M$, then the truncation function in the equation \eqref{eq:phi_lambda} and the modulus in equation \eqref{eq:lactate_lambda} can be removed.

\begin{prop} \label{prop:existence_lambda}
    For all $\lambda \in (0,\lambda^*)$, the approximate PDE system \eqref{eq:problem_lambda} endowed with the boundary and initial conditions \eqref{eq:boundary_conditions}--\eqref{eq:initial_conditions} admits at least one weak solution $(\varphi_{\lambda},\sigma_{\lambda},\uu_{\lambda},z_{\lambda})$ in the sense of \Cref{defn:weak_solution_approximate}. Moreover, we have that
    \begin{equation} \label{eq:aprox_syst_phi_sigma_boundedness}
        0 \leq \varphi_{\lambda} \leq N, \qquad 0 \leq \sigma_{\lambda} \leq M 
    \end{equation}
    a.e. in $Q$, where $M=\max\{M_0, J^*\}e^T$. Finally, there exists a positive constant $C$ depending only on the data of the problem and not depending on $\lambda$ such that
    \begin{align}
         \norm{\varphi_{\lambda}}_{H^1(\Hs)\cap L^{\infty}(\Vs) \cap L^2(\Ws)} &\leq C,\label{eq:schauder_phil_estimate}\\
         \norm{\sigma_{\lambda}}_{H^1(\Vs')\cap L^{\infty}(\Hs) \cap L^2(\Vs)} &\leq C,\label{eq:schauder_sigmal_estimate}\\
         \norm{\uu_{\lambda}}_{W^{1,\infty}(\Vs_0)}&\leq C,\label{eq:schauder_ul_estimate}\\
         \norm{z_{\lambda}}_{H^1(\Hs)\cap L^{\infty}(\Vs) \cap L^2(\Ws)}&\leq C,\label{eq:schauder_zl_estimate}\\
          \norm{\beta_{\lambda}(z_{\lambda})}_{ L^2(\Hs)}&\leq C.\label{eq:schauder_betal_estimate}
    \end{align}
\end{prop}

\begin{proof}
    The proof is based on the Schauder fixed-point argument (see e.g. \cite[][p. 179]{brezis2011}). We introduce the Banach space
    \begin{equation*}
        \mathcal{X} \coloneqq \{(\sigma,z) \in   L^2(0,T;\Hs) \times L^2(0,T;L^{\infty})\}
    \end{equation*}
    endowed with the standard norm
    \begin{equation*}
        \norm{(\sigma,z)}_{\mathcal{X}} \coloneq \norm{\sigma}_{L^2(\Hs)} + \norm{z}_{L^2(L^{\infty})}.
    \end{equation*}
    In what follows, we construct an operator $\gamma : \mathcal{X} \to \mathcal{X}$, to which we intend to apply the Schauder fixed-point argument. \\

    \noindent \textit{Step 1. } Starting from $( \overline{\sigma}, \overline{z}) \in \mathcal{X}$, we find $\varphi \in H^1(0,T;\Hs) \cap L^{\infty}(0,T; \Vs) \cap L^{2}(0,T;\Ws)$ as the unique solution of the semilinear parabolic equation with Lipschitz continuous nonlinearity
\begin{equation}\label{eq:schauder_phi}
        \begin{cases}
        \partial_t \varphi -\Delta \varphi = p(\overline{\sigma}, \overline{z}) \alpha(\varphi) - \varphi g(\overline{\sigma}, \overline{z}) & \text{ in } Q,\\
        \partial_{\nnu} \varphi = 0 & \text{ on } \Sigma,\\
        \varphi(0) = \varphi_0 & \text{ in } \Omega,
    \end{cases}
\end{equation}
exploiting that $g(\overline{\sigma}, \overline{z}) \in L^{\infty}(Q)$ and $p(\overline{\sigma}, \overline{z}) \in L^{\infty}(Q)$. The well-posedness of this system can be proved in several classical ways, such as using Galerkin discretization (see, e.g., \cite[Lemma 5.3, p. 373]{Troltzsch_2010}) or semigroup theory (see, e.g., \cite[Chapter 20]{pata19}). The regularity can be shown by standard results for linear parabolic equations (see \cite{Dautray_Lions_92, Lions61}), from which we obtain the estimate
    \begin{equation}\label{eq:schauder_phi_estimate}
        \norm{\varphi}_{H^1(\Hs)\cap L^{\infty}(\Vs) \cap L^2(\Ws)} \leq C,
    \end{equation}
    for a certain positive constant $C$ independent from $\lambda$ and $(\overline{\sigma},\overline{z})$. Next, we aim to prove that 
\begin{equation}\label{eq:schauder_phi_N}
    0 \leq \varphi \leq N.
\end{equation}
To do so, we test the first equation of the system \eqref{eq:schauder_phi} with $(\varphi-N)^+$ and we integrate over $\Omega$. Thanks to the boundary condition, we obtain:
\begin{align*}
    \frac{1}{2} \frac{\textrm{d}}{\dt}\into |(\varphi-N)^+|^2 \dx &\leq \frac{1}{2} \frac{\textrm{d}}{\dt}\into |(\varphi-N)^+|^2 \dx + \into |\nabla [(\varphi-N)^+]|^2 \dx\\
    &= \into p(\overline{\varphi},\overline{z}) \alpha(\varphi)(\varphi-N)^+ \dx - \into \varphi g(\overline{\sigma},\overline{z})(\varphi-N)^+ \dx \leq 0,
\end{align*}
    where the last inequality holds because where $\varphi \geq N$ by definition $\alpha(\varphi)=0$ (so the first addend is equal to 0) and $\varphi$ is trivially positive (so the second addend is non-positive). Integrating in time, it follows that
    \begin{equation*}
        \into |(\varphi-N)^+|^2 \dx \leq \into |(\varphi_0-N)^+|^2 \dx = 0,
    \end{equation*}
    employing the Hypothesis \ref{hyp:initial_conditions} according to which $\varphi_0 \leq N$. As a consequence, $\varphi \leq N$ a.e. in $Q$. In a very similar way, we test the same equation by $-\varphi^-$ and integrate over $\Omega$, obtaining:
    \begin{align*}
        \frac{1}{2} \frac{\textrm{d}}{\dt} \into |\varphi^-|^2 \dx &\leq \frac{1}{2} \frac{\textrm{d}}{\dt} \into |\varphi^-|^2 \dx + \into |\nabla \varphi^-| \dx\\
        &= - \into p(\overline{\sigma},\overline{z})\alpha(\varphi)\varphi^- \dx - \into |\varphi^-|^2 g(\overline{\sigma},\overline{z}) \dx \leq 0.
    \end{align*}
    Integrating in time, it follows that
    \begin{equation*}
        \into |\varphi^-|^2 \dx \leq \into |\varphi_0^-|^2 \dx=0,
    \end{equation*}
    since $\varphi_0 \geq 0$ by Hypothesis \ref{hyp:initial_conditions}, so $\varphi \geq 0$ a.e. in $Q$. Notice that we can consequentially remove the truncation $\alpha$.\\

     \noindent \textit{Step 2.} Starting from $(\overline{\sigma},\overline{z})$ and $\varphi$, we find $\sigma \in  H^1(0,T;\Vs') \cap L^{\infty}(0,T;\Hs) 
    \cap L^2(0,T;\Vs)$ as the unique solution of the following linear parabolic equation
     \begin{equation}\label{eq:schauder_sigma}
         \begin{cases}
             \displaystyle  \partial_t \sigma -\Delta\sigma + \frac{k_1(\varphi,\overline{z})\sigma}{k_2(\varphi,\overline{z}) + |\overline{\sigma}|} = J(\varphi,\overline{z}) & \text{in } Q,\\
             \partial_{\nnu}\sigma + \sigma -\sigma_{\Gamma} = 0 & \text{on } \Sigma,\\
             \sigma(0)=\sigma_0 & \text{in } \Omega.
         \end{cases}
     \end{equation}
     Using Hypothesis \ref{hyp:eq_lactate} and standard regularity results (see, e.g., \cite{Lions61}), we also have that
     \begin{equation}\label{eq:schauder_sigma_estimate}
         \norm{\sigma}_{H^1(\Vs') \cap L^{\infty}(\Hs) \cap L^2(\Vs)}\leq C,
     \end{equation}
    for a positive constant $C$ that does not depend on $\lambda$, $(\overline{\sigma},\overline{z})$ and $\varphi$. Now we want to prove that there exists a positive constant $M$ that depends only on $M_0$, $J^*$ and $T$ such that
    \begin{equation}\label{eq:schauder_sigma_M}
        0 \leq \sigma \leq M
    \end{equation}
    almost everywhere in $Q$.
    Multiplying the first equation in \eqref{eq:schauder_sigma} with $-\sigma^-$, integrating over $\Omega$ and employing the boundary condition, we get
    \begin{align*}
        \frac{1}{2} &\frac{\text{d}}{\dt} \into |\sigma^-|^2 \dx\\
        &\leq \frac{1}{2} \frac{\text{d}}{\dt} \into |\sigma^-|^2 \dx + \into |\nabla \sigma^-|^2 \dx + \intg |\sigma^-|^2 \dS + \into \frac{k_1(\varphi,\overline{z})|\sigma^-|^2}{k_2(\varphi,\overline{z}) + |\overline{\sigma}|} \dx\\
        &= - \into J(\varphi,\overline{z}) \sigma^- \dx - \intg \sigma_{\Gamma}\sigma^- \dS \leq 0. 
    \end{align*}
    Integrating in time over $(0,t)$, we have
    \begin{equation*}
        \frac{1}{2}\into|\sigma^-|^2 \dt\leq \frac{1}{2}\into|\sigma_0^-|^2 \dt = 0,
    \end{equation*}
    where the last equality stands because of Hypothesis \ref{hyp:initial_conditions}. As a consequence, $\sigma \geq 0$ a.e. in $Q$. Finally, we employ a standard Moser--Alikakos technique to prove that $\sigma \leq M$ for a certain $M>0$ yet to be found. We multiply the first equation in \eqref{eq:schauder_sigma} by $q\sigma^{q-1}$ with $q>2$ and integrating in space over $\Omega$. Notice that, in order to be sure that all the integrals are well-defined, one should introduce a truncation of $\sigma$,
    \begin{equation*}
        \sigma_k \coloneq \begin{cases}
            \sigma & \text{ if } \sigma \leq k\\
            k & \text{otherwise,}
        \end{cases}
    \end{equation*}
    for $k \in \NN$, multiply the previous equation by $q(\sigma_k)^{q-1}$ and proceed as we will do. In the end, having obtained an estimate that does not depend on $k$, one should pass to the limit as $k \to +\infty$ and recover the thesis. We will not do it in this rigorous way to avoid overloading the notation and we proceed formally testing the first equation in \eqref{eq:schauder_sigma} by $q \sigma^{q-1}$, obtaining
    \begin{align*}
        \frac{\text{d}}{\dt}& \into \sigma^q \dx + q(q-1)\into \sigma^{q-2}|\nabla \sigma|^2 \dx + q \intg \sigma^q \dS + q \into \frac{k_1(\varphi,\overline{z})\sigma^q}{k_2(\varphi,\overline{z})+ |\overline{\sigma}|} \dx \\
        &= q \intg \sigmag\sigma^{q-1} \dS + q \into J(\varphi,\overline{z}) \sigma^{q-1} \dx \leq q M_0 \intg \sigma^{q-1} \dS + q J^* \into \sigma^{q-1} \dx\\
        & \leq (q-1) \intg \sigma^q \dS + (M_0)^q |\Gamma| + (q-1) \into \sigma^q \dx + (J^*)^q |\Omega|
    \end{align*}
    having used the boundedness of $\sigmag$ and $J$ from Hypotheses \ref{hyp:boundary_conditions} and \ref{hyp:eq_lactate} and then the Young inequality with exponents $q/(q-1)$ and $q$. Doing the obvious simplification in the previous inequality and employing the fact that $\sigma$ is non-negative, we obtain:
    \begin{align*}
        \frac{\text{d}}{\dt}& \into \sigma^q \dx\\
        &\leq \frac{\text{d}}{\dt} \into \sigma^q \dx + q(q-1)\into \sigma^{q-2}|\nabla \sigma|^2 \dx +  \intg \sigma^q \dS + q \into \frac{k_1(\varphi,\overline{z})\sigma^q}{k_2(\varphi,\overline{z})+ |\overline{\sigma}|} \dx\\
        & \leq (M_0)^q |\Gamma| + (q-1) \into \sigma^q \dx + (J^*)^q |\Omega|.
    \end{align*}
    Integrating with respect to time over $(0,t)$, knowing that by Hypothesis \ref{hyp:initial_conditions} $\sigma_0 \leq M_0$, we find
    \begin{align*}
        \into \sigma^q(t) \dx &\leq \into \sigma_0^q \dx + (M_0)^q |\Gamma| T + (J^*)^q |\Omega|T + (q-1) \intto \sigma^q \dx \dt\\
        &\leq (M_0)^q (|\Omega| + |\Gamma|T) + (J^*)^q|\Omega|T +  (q-1) \intto \sigma^q \dx \dt,
    \end{align*}
    from which we deduce
    \begin{equation*}
        \norm{\sigma(t)}_{L^q} \leq \left[ M_0(|\Omega| + |\Gamma|T)^{\frac{1}{q}} + J^*|\Omega|^{\frac{1}{q}}T^{\frac{1}{q}} \right] e^{\frac{q-1}{q}T} \eqcolon C_q
    \end{equation*}
    using Gronwall inequality and then taking the $q^{th}$-root of both sides of the inequality. Passing to the limit as $q \to +\infty$, observing that by easy calculations $C_q \to M_0 e^T$ or $C_q \to J^*e^T$, we have
    \begin{equation*}
        \norm{\sigma(t)}_{L^\infty} \leq M \coloneq \max\{M_0,J^*\}e^T
    \end{equation*}
    for a.e. $t \in (0,T)$.\\
    
    \noindent \textit{Step 3.} Starting from $\varphi$ and $\overline{z}$, we find $\uu \in W^{1,\infty}(0,T;\Vs_0)$ as the unique solution of 
    \begin{equation}\label{eq:schauder_u}
        \begin{cases}
            \displaystyle \into \left[\A(\varphi,\overline{z})\epsilonut + \B(\varphi,\overline{z})\epsilonu\right] : \epsilonv \dx  = \into \ff \cdot \vv \dx & \forall \vv \in \Vs_0,\\
            \uu(0)=\uu_0 & \text{ a.e. in } \Omega.
        \end{cases}
    \end{equation}
    To do so, we proceed by time discretization. We consider a uniform partition of $[0,T]$ with  time step $\tau > 0$ and equidistant nodes $0= t_0 < t_1 < \dots < t_{K_{\tau}}=T$. We also introduce the notation: 
    \begin{equation*}
        I_{\tau}^k \coloneqq \begin{cases}
            [0,\tau] &\text{ if } k=1,\\
            (t_{\tau}^{k-1},\ t_{\tau}^k] &\text{ if } k = 2 \dots K_{\tau}.
        \end{cases}
    \end{equation*}
    We approximate $\varphi$, $\overline{z}$, and $f$ with their local means, i.e., we define
    \begin{equation*}
        \fk \coloneqq \frac{1}{\tau}\intk \ff\dd s, \qquad  \phik \coloneqq \frac{1}{\tau}\intk \varphi\dd s,\qquad  \zk \coloneqq \frac{1}{\tau}\intk \overline{z} \dd s,  
    \end{equation*}
    for every $k=1, \dots, K_{\tau}$. To keep the notation short, we also introduce
    \begin{equation*}
        \Ak \coloneqq \A(\phik,\zk), \qquad \Bk \coloneqq \B(\phik,\zk). 
    \end{equation*}
    \begin{remark}
    It is obvious that, since $\varphi, \overline{z} \in L^{\infty}(0,T;\Vs)$ and  $\ff \in L^{\infty}(0,T;\Hs)$, then 
    $\phik, \zk \in \Vs$ and $\fk \in \Hs$ with
    \begin{equation} \label{eq:estimate_dataK}
        \norm{\phik}_{\Vs} \leq \norm{\varphi}_{L^{\infty}(\Vs)},\quad \norm{\zk}_{\Vs} \leq \norm{\overline{z}}_{L^{\infty}(\Vs)},\quad
        \norm{\fk}_{\Hs} \leq \norm{f}_{L^{\infty}(\Hs)},
    \end{equation}
    for every $k=1, \dots, K_{\tau}$.
    \end{remark}
    \noindent Starting from $\uu^0_{\tau} = \uu_0 \in \Vs_0$, we solve recursively the following time-discrete problem:
    \begin{equation}\label{eq:discrete_problem_u}
        \begin{cases}
            \text{Given } \ukm \in \Vs_0, \text{ find } \uk \in \Vs_0 \text{ s.t. for all } \vv \in \Vs_0\\
            \displaystyle \into \Ak \varepsilon\bigg(\frac{\uk-\ukm}{\tau}\bigg): \epsilonv \dx  + \into \Bk \epsilonuk : \epsilonv \dx = \into \fk \cdot \vv \dx.
        \end{cases}
    \end{equation}
    Since the equation in \eqref{eq:discrete_problem_u} can be rewritten as
    \begin{equation}\label{eq:discrete_problem_u_rewritten}
        \into (\Ak + \tau \Bk) \epsilonuk : \epsilonv \dx = \into \tau \fk \cdot \vv + \Ak \epsilonukm : \epsilonv \dx  
    \end{equation}
    where $\Ak + \tau \Bk$ is strictly positive definite and the right-hand side yields a linear bounded functional on $\Vs_0'$, the existence follows from Lax--Milgram Theorem. Now our aim is to pass to the limit in the discrete equation, recovering a solution to the original problem. For the sake of brevity, we introduce the shorter notation
    \begin{equation*}
        \vk \coloneqq \frac{\uk -\ukm}{\tau}
    \end{equation*}
    to denote the discrete velocity. Taking $\vk$ as a test function in equation \eqref{eq:discrete_problem_u}, it is easy to prove that
    \begin{equation}\label{eq:estimate_discrete_u_1}
        \norm{\vk}_{\Vs_0} \leq C\left(\norm{\ff}_{L^{\infty}(\Hs)} + \norm{\uu_0}_{\Vs_0} \right), 
    \end{equation}
    where the constant $C$ depends on $T$ but non on $k$, $\lambda$, $\phik$, and $\zk$.
    As a consequence, since
    \begin{equation*}
        \uk = \tau \sum_{j=1}^k \vj + \uu_0,
    \end{equation*}
    we also have that 
    \begin{equation}\label{eq:estimate_discrete_u_2}
        \norm{\uk}_{\Vs_0} \leq C
    \end{equation}
    where, again $C$ depends on $\ff$, $\uu_0$ and $T$ but not on $k$, $\lambda$, $\phik$, and $\zk$.
    Given any sequence of scalar, vector-valued or tensor-value functions $\{w_{\tau}^k\}_{k=0}^{K_{\tau}}$ defined over $\Omega$, we introduce the piecewise constant interpolations $w_{\tau}$ and the piecewise linear interpolation $\widehat{w}_{\tau}$ over the time interval $[0,T]$ as
    \begin{equation*}
        w_{\tau}(t)\coloneqq w_{\tau}^k,\qquad
        \widehat{w}_{\tau}(t) \coloneqq \frac{t - t_{\tau}^{k-1}}{\tau} w_{\tau}^k + \frac{t_{\tau}^k -t}{\tau} w_{\tau}^{k-1}
    \end{equation*}
    for every $t \in I_{\tau}^k$. With this new notation, notice that $\partial_t \uti = \vt$ and  estimates \eqref{eq:estimate_discrete_u_1}, \eqref{eq:estimate_discrete_u_2} trivially lead to
    \begin{equation}\label{eq:estimate_discrete_u_3}
        \norm{\uti}_{W^{1,\infty}(\Vs_0)} + \norm{\ut}_{L^{\infty}(\Vs_0)}  \leq C.
    \end{equation}
    By standard compactness results, we deduce that there exists a $\uu \in W^{1,\infty}(0,T;\Vs_0)$ such that 
    \begin{alignat}{2}
        & \uti \to \uu \qquad \text{weakly-}\ast &&\qquad\text{in }  W^{1,\infty}(0,T;\Vs_0),\\
        & \ut \to \uu \qquad \text{weakly-}\ast &&\qquad \text{in }  L^{\infty}(0,T;\Vs_0).
    \end{alignat}
    Moreover, by their definition as time local means value, it holds true that
    \begin{equation}
        \ft \to \ff \qquad \text{strongly}\qquad  \text{in } L^2(0,T;\Hs)
    \end{equation}
    and that
    \begin{equation}
        \phit \to \varphi, \,\, \zt \to \overline{z} \qquad \text{a.e.}\qquad \text{in } Q.
    \end{equation}
   Recalling that $\A$ and $\B$ are continuous and bounded by Hypothesis \ref{hyp:eq_displacement}, it follows that
   \begin{equation}
       \At \to \A(\varphi,\overline{z}), \,\, \Bt \to \B(\varphi,\overline{z}) \qquad \text{ strongly} \qquad \text{in } L^2(0,T;\Hs)
   \end{equation}
    by the Dominated Convergence Theorem. These convergences we proved are enough to pass to the limit in the equation
    \begin{equation*}
         \into \At \varepsilon(\vt): \epsilonv \dx  + \into \Bt \varepsilon(\ut) : \epsilonv \dx = \into \ft \cdot \vv \dx
    \end{equation*}
    for every $\vv \in \Vs_0$ and a.e. $t \in (0,T)$, showing that $\uu$ is a solution to the original system. Notice that from lower semi-continuity of the norm with respect to weak-$\ast$ convergence and estimate \eqref{eq:estimate_discrete_u_3}, we also have that
    \begin{equation}\label{eq:schauder_u_estimate-bis}
        \norm{\uu}_{W^{1,\infty}(\Vs_0)} \leq C,
    \end{equation}
    where $C$ depends on $T$, $\ff$, and $\uu_0$ but it is independent of $\lambda$ and $(\varphi,\overline{z})$.
    Finally, we should prove that the solution is unique, but this comes easily from the fact that the equation is linear in $\uu$, $\A$ is strictly positive definite, and $\B$ is bounded.\\
    
    \noindent \textit{Step 4.} Starting from $\varphi$ and $\uu$, we find $z \in H^1(0,T;\Hs) \cap L^{\infty}(0,T;\Vs) \cap L^2(0,T;\Ws)$ as the unique solution of the semilinear parabolic equation with Lipschitz continuous nonlinearity
    \begin{equation}\label{eq:schauder_z}
        \begin{cases}
            \partial_t z -\Delta z + \beta_{\lambda}(z) + \pi(z) = w - \Psi(\varphi,\epsilonu) & \text{in Q},\\
            \partial_{\nnu}z = 0 & \text{on } \Sigma,\\
            z(0)=z_0 &\text{in } \Omega.
        \end{cases}
    \end{equation}
    Notice that, thanks to Hypothesis \ref{hyp:w},  inequality \eqref{eq:ineq_psi}, and estimates \eqref{eq:schauder_phi_estimate}, \eqref{eq:schauder_u_estimate-bis}, $w-\Psi(\varphi,\epsilonu)$ is uniformly bounded in $L^{\infty}(0,T;\Hs)$. Now we want to prove that there exists a positive constant $C$ (which does not depend on $\lambda$ and $(\overline{\sigma},\overline{z})$) such that
    \begin{equation}\label{eq:schauder_z_estimate}
        \norm{z}_{H^1(\Hs) \cap L^{\infty}(\Vs) \cap L^2(\Ws) } \leq C.
    \end{equation}
    Testing the first equation in \eqref{eq:schauder_z} by $\partial_t z$ and integrating over $\Omega$, we have
    \begin{align*}
        \into &|\partial_t z|^2 \dx + \frac{1}{2} \frac{\text{d}}{\dt} \into |\nabla z|^2 \dx + \frac{\text{d}}{\dt} \into \widehat{\beta}_{\lambda}(z) \dx \\
        &= - \into \pi(z)\partial_t z\dx + \into \left[w - \Psi(\varphi,\epsilonu) \right]\partial_t z\dx\\
        &\leq C \into \left( |z| + 1 + |w| + |\varphi| + |\epsilonu| +  |\hat{\Psi}|\right) |\partial_t z| \dx\\
        &\leq \frac{1}{2}\into |\partial_t z|^2 \dx + C \into \left(|z|^2 + 1 + |w|^2 + |\varphi|^2 + |\epsilonu|^2 + |\hat{\Psi}|^2 \right) \dx\\
        &\leq \frac{1}{2}\into |\partial_t z|^2 \dx + C \into |z_0|^2 \dx + C\intto |\partial_t z|^2 \dx \ds\\
        &\quad + C \into \left( 1 + |w|^2 + |\varphi|^2 + |\epsilonu|^2 + |\hat{\Psi}|^2 \right) \dx,
    \end{align*}
    where we have used the fact that $\pi$ is  Lipschitz continuous by Hypothesis \ref{hyp:pi}, inequality \eqref{eq:ineq_psi} (cf. Hypothesis \ref{hyp:psi}), applying Young inequality and the fact that
    \begin{equation*}
        z(t)= z_0 + \intt \partial_t z \ds.
    \end{equation*}
    Integrating in time over $(0,\tau)$, we obtain
    \begin{align*}
        \frac{1}{2} &\inttauo |\partial_t z|^2 \dx \!\dt + \frac{1}{2} \into |\nabla z|^2 \dx + \into \widehat{\beta}_{\lambda}(z) \dx\\
        &\leq \frac{1}{2} \into |\nabla z_0|^2 \dx + \into \widehat{\beta}_{\lambda}(z_0) \dx + C \into |z_0|^2 \dx  + C\inttau \! \! \! \intto |\partial_t z|^2 \dx \! \ds  \!\dt\\
        &\quad + C \inttauo  \left(1 + |w|^2 + |\varphi|^2 + |\epsilonu|^2 + |\hat{\Psi}|^2 \right) \dx \!\dt\\
        &\leq C_0 + C + C\inttau \! \! \! \intto |\partial_t z|^2 \dx \! \ds  \!\dt,
    \end{align*}
    where we have used the fact that $\widehat{\beta}_{\lambda} \leq \widehat{\beta}$ by definition of Yosida approximation (cf. \cite{brezis1973}). Applying Gronwall inequality and the fact that $z$ is bounded, we obtain
    \begin{equation*}
        \norm{z}_{H^1(\Hs) \cap L^{\infty}(\Vs)} \leq C,
    \end{equation*}
    where $C>0$ does not depend on $(\overline{\sigma},\overline{z})$. By comparison in the first equation in \eqref{eq:schauder_z}, we have 
    \begin{equation*}
        \begin{split}
            \norm{-\Delta z + \beta_{\lambda}(z)}_{L^2(\Hs)} &= \norm{- \partial_t z - \pi(z) + w -\Psi(\varphi,\epsilonu)}_{L^2(\Hs)}\\
            &\leq C \left(\norm{ z }_{H^1(\Hs)} + 1 +  \norm{w- \Psi(\varphi,\epsilonu)}_{L^{\infty}(\Hs)}\right)\leq C
        \end{split}
    \end{equation*}
    where $C$ does not depend on $\lambda$ and $(\overline{\sigma},\overline{z})$.
    On the other hand, we observe that
    \begin{equation*}
        \begin{split}
             \norm{-\Delta z + \beta_{\lambda}(z)}_{L^2(\Hs)}^2 &= \norm{-\Delta z}_{L^2(\Hs)}^2 + \norm{\beta_{\lambda}(z)}_{L^2(\Hs)}^2 + 2\intto -\Delta z \,\beta_{\lambda}(z) \dx \ds\\
            &= \norm{-\Delta z}_{L^2(\Hs)}^2 + \norm{\beta_{\lambda}(z)}_{L^2(\Hs)}^2 + 2\intto  \,\beta_{\lambda}'(z) |\nabla z|^2 \dx \ds\\
            & \geq \norm{-\Delta z}_{L^2(\Hs)}^2 + \norm{\beta_{\lambda}(z)}_{L^2(\Hs)}^2,
        \end{split}
    \end{equation*}
    where the inequality holds because $\beta_{\lambda}'$ is monotone and Lipschitz continuous (so it is a.e. differentiable with non-negative derivative). Putting together the previous two inequalities, we have proved 
    \begin{equation*}
        \norm{-\Delta z}_{L^2(\Hs)} + \norm{\beta_{\lambda}(z)}_{L^2(\Hs)}\leq C
    \end{equation*}
    and, as a consequence, we finally obtain the estimate \eqref{eq:schauder_z_estimate}.\\

    \noindent In the previous steps, we have built an operator $\gamma: \mathcal{X} \to \mathcal{X}$ such that $\gamma(\overline{\sigma},\overline{z}) \coloneq (\sigma,z)$. 
    From what we have already proved, it is straightforward that
    \begin{center}
        $\gamma$ is well-defined,
    \end{center}
    because each of the problems \eqref{eq:schauder_phi}, \eqref{eq:schauder_sigma}, \eqref{eq:schauder_u}, \eqref{eq:schauder_z} is well-posed, and that
    \begin{center}
        $\gamma(\mathcal{X})$ is a compact subset of $\mathcal{X}$,
    \end{center}
    because the following compact embeddings hold
    \begin{alignat*}{2}
        &H^1(0,T;V') \cap L^{\infty}(0,T;\Hs) \cap L^2(0,T;\Vs)&&\hookrightarrow \hookrightarrow \,\,L^2(0,T;\Hs),\\ 
        & H^1(0,T;\Hs) \cap L^{\infty}(0,T;V)\cap L^2(0,T;\Ws) &&\hookrightarrow \hookrightarrow \,\,L^2(0,T;L^{\infty})
    \end{alignat*}
    by Aubin--Lions Theorem (see \cite[Section 8, Corollary 4]{Simon_86}). To apply the Schauder fixed point Theorem, it remains to prove that
    \begin{center}
        $\gamma$ is continuous with respect to the norm $\norm{\cdot}_{\mathcal{X}}$.
    \end{center}
    Thus, given a sequence $(\overline{\sigma}_k,\overline{z}_k)$ strongly converging to $(\overline{\sigma},\overline{z})$ in $\mathcal{X}$, i.e. such that
    \begin{alignat}{2}
        \overline{\sigma}_k \to \overline{\sigma} &\qquad \text{strongly}&&\qquad \text{in }L^2(0,T;\Hs) \label{eq:schauder_osigma_strong},\\
        \overline{z}_k \to \overline{z} &\qquad \text{strongly}&&\qquad \text{in }L^{2}(0,T;L^{\infty}) \label{eq:schauder_oz_strong},
    \end{alignat}
    we aim to verify that $(\sigma_k,z_k) \to (\sigma,z)$ strongly in $\mathcal{X}$. By the uniform estimates \eqref{eq:schauder_sigma_estimate}, \eqref{eq:schauder_z_estimate} and standard compactness results, we know that there exists a pair $(\rho,\zeta)$ such that, along a subsequence that we do not relabel, 
    \begin{alignat}{2}
        \sigma_k \to \rho &\qquad \text{weakly-}\ast&& \qquad\text{in }H^1(0,T;V') \cap L^{\infty}(0,T;\Hs) \cap L^2(0,T;\Vs),\\
        &\qquad \text{strongly}&& \qquad \text{in } L^2(0,T;\Hs),\label{eq:schauder_sigma_strong}\\
         &\qquad \text{a.e.}&& \qquad\text{in } Q,\label{eq:schauder_sigma_ae}\\
        z_k \to \zeta &\qquad \text{weakly-}\ast&& \qquad \text{in } H^1(0,T;H) \cap L^{\infty}(0,T;\Vs) \cap L^2(0,T;\Ws),\label{eq:schauder_z_weak}\\
        &\qquad \text{strongly}&& \qquad\text{in } L^2(0,T;L^{\infty}),\label{eq:schauder_z_strong}\\
        &\qquad \text{a.e.}&& \qquad\text{in } Q. \label{eq:schauder_z_ae}
    \end{alignat}
    The proof is complete if we show that $(\rho,\zeta)=(\sigma,z)$, i.e. that $\rho$ (resp. $\zeta$) is the solution of problem \eqref{eq:schauder_sigma} (resp. \eqref{eq:schauder_z}) corresponding to the datum $(\overline{\sigma},\overline{z})$. In fact, at this point, since every subsequence admits a sub-subsequence that converges to the same limit, the convergences \eqref{eq:schauder_sigma_strong} and \eqref{eq:schauder_z_strong} hold for the whole sequence. To do so, we pass to the limit as $k \to +\infty$ in the systems \eqref{eq:schauder_phi}, \eqref{eq:schauder_sigma}, \eqref{eq:schauder_u}, \eqref{eq:schauder_z} with initial datum $(\overline{\sigma}_k,\overline{z}_k)$ that, by definition of $\gamma$, are satisfied by $\varphi_k$, $\sigma_k$, $\uu_k$ and $z_k$.\\

    \noindent\textit{Step I.} We know that $\varphi_k$ satisfies 
    \begin{equation*}
        \partial_t \varphi_k -\Delta \varphi_k = p(\overline{\sigma}_k, \overline{z}_k) \alpha(\varphi_k) - \varphi_k g(\overline{\sigma}_k, \overline{z}_k)
    \end{equation*}
    in $L^2(0,T;\Hs)$ and we aim to pass to the weak limit in this equation as $k \to +\infty$. From the uniform estimate \eqref{eq:schauder_phi_estimate} and standard compactness result, we assert that it exists a $\phi$ such as, along a further subsequence that we do not relabel,
    \begin{alignat}{2}
        \varphi_k \to \phi &\qquad \text{weakly-}\ast&& \qquad\text{in } H^1(0,T;H) \cap L^{\infty}(0,T;\Vs) \cap L^2(0,T;\Ws),\\
        &\qquad \text{strongly}&& \qquad \text{in } C^0([0,T];H^{1-\epsilon})\cap L^2(0,T;H^{2-\epsilon}) \text{ for } 0<\epsilon<1, \label{eq:schauder_phi_strong}\\
         &\qquad \text{a.e}&& \qquad \text{in } Q. \label{eq:schauder_phi_ae}
    \end{alignat}
    Moreover, since $\overline{\sigma}_k \to \overline{\sigma}$ strongly in $L^2(0,T;\Hs)$ and $\overline{z}_k \to \overline{z}$ strongly in $L^2(0,T;L^{\infty})$, it holds 
    \begin{equation}\label{eq:schauder_z_sigma_ae}
        \overline{\sigma}_k \to \overline{\sigma}, \, \, \overline{z}_k \to \overline{z} \qquad \text{a.e} \qquad \text{in }  Q,
    \end{equation}
    possibly extracting a further subsequence. 
    This implies that we can pass to the limit in the equation: the terms on the left-hand side are trivial and the terms on the right-hand side converge strongly in $L^2(0,T;\Hs)$ because we can apply the Lebesgue Convergence Theorem. In fact $p$, $\alpha$ and $g$ are continuous and uniformly bounded (thanks to Hypothesis \ref{hyp:eq_phi} and by definition of $\alpha$ in \eqref{eq:alpha}) and their arguments converge a.e. (thanks to \eqref{eq:schauder_phi_ae}, \eqref{eq:schauder_z_sigma_ae}). 
    So, $\phi$ actually solves system \eqref{eq:schauder_phi} with data $(\overline{\sigma},\overline{z})$: by uniqueness of the solution, we may identify $\phi$ with $\varphi$.\\

    \noindent \textit{Step II}. The passage to the limit in the equation
    \begin{equation*}
        \partial_t \sigma_k -\Delta \sigma_k + \frac{k_1(\varphi_k,\overline{z}_k)}{k_2(\varphi_k,\overline{z}_k)+ |\overline{\sigma}_k|} = J(\varphi_k,\overline{z}_k)
    \end{equation*}
    is quite similar to the previous one so we will not do it in detail. It allows us to assert that $\rho=\sigma$ and satisfies \eqref{eq:schauder_sigma} for the data $\varphi$ and $(\overline{\sigma},\overline{z})$.\\

    \noindent \textit{Step III.}  Regarding $\uu_k$, the solution of \eqref{eq:schauder_u} from the data  $(\varphi_k,\overline{z}_k)$, it exists a $\oomega$ such that, along a non-relabeled subsequence,
    \begin{align} \label{eq:schauder_u_weak}
        \uu_k \to \oomega \qquad \text{weakly-}\ast \qquad \text{in } W^{1,\infty}(0,T;\Vs_0)
    \end{align}
    from \eqref{eq:schauder_u_estimate-bis} and standard compactness results. Thanks to this convergence, we are able to pass to the limit in the equation
    \begin{equation*}
        \into \A(\varphi_k,\overline{z}_k)\varepsilon(\partial_t\uu_k) : \epsilonv \dx + \into \B(\varphi_k,\overline{z}_k)\varepsilon(\uu_k):\epsilonv \dx = \into \ff \cdot \vv \dx 
    \end{equation*}
    for every fixed $\vv\in \Vs_0$. In fact, $\mathcal{A}$ and $\mathcal{B}$ are continuous from Hypothesis \ref{hyp:eq_displacement} and $\varphi_k$, $\overline{z}_k$ converge a.e. to $\varphi$, $\overline{z}$ employing \eqref{eq:schauder_phi_ae} and \eqref{eq:schauder_z_sigma_ae}, which implies that
    \begin{equation*}
        \mathcal{A}(\varphi_k,\overline{z}_k) \to \mathcal{A}(\varphi,\overline{z}), \qquad \mathcal{B}(\varphi_k,\overline{z}_k) \to \mathcal{B}(\varphi,\overline{z})
    \end{equation*}
    a.e. in $Q$. Moreover, $\mathcal{A}$ and $\mathcal{B}$ are bounded by Hypothesis \ref{hyp:eq_displacement}. Thus, \eqref{eq:schauder_u_weak} is enough to pass to the limit in the previous equation, and we may identify $\oomega$ with $\uu$ since it is the unique solution of the system \eqref{eq:schauder_u} with initial data $\varphi$ and $\overline{z}$.\\
    
    \noindent \textit{Step IV.} The passage to the limit in the equation
    \begin{equation*}
        \partial_t z_k -\Delta z_k + \beta_{\lambda}(z_k) + \pi(z_k) = w - \Psi(\varphi_k,\varepsilon(\uu_k)),
    \end{equation*}
    is obvious in all the terms with the exception of $\Psi(\varphi_k, \varepsilon(\uu_k))$, because of the convergence \eqref{eq:schauder_z_weak}, \eqref{eq:schauder_z_strong} and the Lipschitz continuity of $\beta_{\lambda}$ and $\pi$. However, we need to prove stronger convergence for $\varepsilon(\uu_k)$ to treat the last one. Thus, we take the difference of the equations satisfied by $\uu_k$ and $\uu_l$ and test it with $\partial_t \vv$, where $\vv \coloneq \uu_k - \uu_l$. After summing and subtracting some terms, we obtain:
    \begin{multline*}
        \into \A(\varphi_k,\overline{z}_k)\epsilonvt : \epsilonvt \dx = -\into \left[\A(\varphi_k,\overline{z}_k)-\A(\varphi_l,\overline{z}_l)\right]\epsilonult : \epsilonvt \dx\\
        -  \into \B(\varphi_k,\overline{z}_k)\epsilonv:\epsilonvt \dx - \into \left[\B(\varphi_k,\overline{z}_k) - \B(\varphi_l,\overline{z}_l)\right]\epsilonul:\epsilonvt \dx.
    \end{multline*}
    Exploiting strictly positive definiteness of  $\A$, Lipschitz continuity of $\A$ and $\B$, boundedness of $\B$ from Hypothesis \ref{hyp:eq_displacement}, and  using H\"older and Young inequalities, for a.e. $t \in (0,T)$ we get
    \begin{align*}
       C_{{\A}_*}\into &|\epsilonvt|^2 \dx \leq  \eta \into |\epsilonvt|^2 \dx + C_{\eta} \into |\epsilonv|^2 \dx\\
        &+ C_{\eta} \left(\norm{\varphi_k-\varphi_l}_{L^{\infty}(\Omega)}^2+\norm{\overline{z}_k-\overline{z}_l}_{L^{\infty}(\Omega)}^2\right)\into \left(|\epsilonult|^2 + |\epsilonul|^2 \right) \dx
    \end{align*}
    for a positive $\eta$ small enough. Then we integrate in time over the interval $(0,t)$ and move to the left-hand side the term multiplied by the small coefficient $\eta$. Moreover, from \eqref{eq:schauder_u_estimate-bis}, we know that $\norm{\epsilonul}_{\Hs}$ and $\norm{\epsilonult}_{\Hs}$ are uniformly bounded in time and the following equality holds
    \begin{equation*}
        \varepsilon(\vv(s)) = \varepsilon(\vv(0)) + \ints\epsilonvt \dtau = \ints \epsilonvt \dtau,
    \end{equation*}
    where we have used the fact that $\uu_k(0)=\uu_l(0)=\uu_0$.
    Putting all these elements together, the previous inequality becomes
    \begin{multline*}
        \intto |\epsilonvt|^2 \ds \leq C \bigg[ 
       \intt \left( \ints \!\! \into |\epsilonvt|^2 \dx  \!\dtau \! \right) \!  \ds\\
       + \norm{\varphi_k-\varphi_l}_{L^2(L^{\infty})}^2 +\norm{\overline{z}_k-\overline{z}_l}_{L^2(L^{\infty})}^2\bigg].
    \end{multline*}
    Applying Gronwall inequality, we obtain
    \begin{align*}
        \norm{\epsilonvt}_{L^2(\Hs)}^2 \leq C e^{CT} \left[ 
       \norm{\varphi_k-\varphi_l}_{L^2(L^{\infty})}^2 +\norm{\overline{z}_k-\overline{z}_l}_{L^2(L^{\infty})}^2\right]\to 0
    \end{align*}
    as $k,l \to +\infty$ thanks to the strong convergences \eqref{eq:schauder_z_strong} and \eqref{eq:schauder_phi_strong} with $\epsilon$ small enough (so that the embedding $H^{2-\epsilon} \hookrightarrow \hookrightarrow C^0(\overline{\Omega})$ holds). This implies that also $\norm{\epsilonv}_{L^{\infty}(\Hs)}$ vanishes in the limit. So, $\{\uu_k\}$ is a Cauchy sequence in $H^1(0,T;V_0)$ and consequentially converges. Thus we have proved that
    \begin{equation}\label{eq:schauder_u_strong}
        \uu_k \to \uu \qquad \text{strongly} \qquad \text{in }  H^1(0,T;V_0),
    \end{equation}
    where we are able to identify the limit with $\uu$ because of the known convergence \eqref{eq:schauder_u_weak}. Now we can conclude the passage to the limit in the damage equation. In fact, $\Psi$ is Lipschitz continuous by Hypothesis \ref{hyp:psi}, $\varphi_k \to \varphi$ strongly in $L^2(0,T;\Hs)$ by \eqref{eq:schauder_phi_strong} and $\varepsilon(\uu_k) \to \epsilonu$ strongly in $L^2(0,T;\Hs)$ by \eqref{eq:schauder_u_strong}, so $\Psi(\varphi_k,\varepsilon(\uu_k)) \to \Psi(\varphi,\epsilonu)$ strongly in $L^2(0,T;\Hs)$. Finally, since $\zeta$ satisfies \eqref{eq:schauder_z} with input data $\varphi$ and $\uu$, we can do the identification $\zeta = z$.\\

    \noindent Applying the Schauder fixed point Theorem, it follows that there exists a $(\sigma_{\lambda},z_{\lambda})$ such that $(\sigma_{\lambda},z_{\lambda})=\gamma(\sigma_{\lambda},z_{\lambda})$. By construction of $\gamma$, we have proved the existence of a quadruplet $(\varphi_{\lambda},\sigma_{\lambda},\uu_{\lambda}, z_{\lambda})$ that is a weak solution of the approximate problem in the sense of \Cref{defn:weak_solution_approximate}. Moreover, this solution satisfies the uniform estimates we investigated throughout the proof since they did not depend on $\lambda$, so \eqref{eq:schauder_phil_estimate}--\eqref{eq:schauder_betal_estimate} hold. 
\end{proof}

\subsubsection{Conclusion of the proof of \Cref{thm:existence}}
\noindent Let us consider a sequence $\{(\varphi_{\lambda},\sigma_{\lambda},\uu_{\lambda},z_{\lambda})\}_{\lambda}$ of weak solutions of the approximate problem. Now we want to pass to the limit as $\lambda \to 0$. Employing the uniform estimates \eqref{eq:schauder_phil_estimate}, \eqref{eq:schauder_sigmal_estimate}, \eqref{eq:schauder_ul_estimate}, \eqref{eq:schauder_zl_estimate}, \eqref{eq:schauder_betal_estimate}, there exist a quadruplet $(\varphi,\sigma,\uu,z)$ and a $\xi$ such that
\begin{alignat}{2}
        \varphi_{\lambda} \to \varphi  &\qquad \text{weakly-}\ast&& \qquad\text{in }H^1(0,T;\Hs) \cap L^{\infty}(0,T;\Vs) \cap L^2(0,T;\Ws),\\
        \sigma_{\lambda}\to \sigma &\qquad \text{weakly-}\ast&& \qquad\text{in } H^1(0,T;V') \cap L^{\infty}(0,T;\Hs) \cap  L^2(0,T;\Vs),\\
        \uu_{\lambda} \to \uu &\qquad \text{weakly-}\ast&& \qquad\text{in }W^{1,\infty}(0,T;\Vs_0),\\
        z_{\lambda} \to z &\qquad \text{weakly-}\ast&& \qquad \text{in }H^1(0,T;\Hs) \cap L^{\infty}(0,T;\Vs) \cap L^2(0,T;\Ws),\label{eq:schauder_zl_weak}\\
         \beta_{\lambda}(z_{\lambda}) \to \xi &\qquad \text{weakly}&& \qquad \text{in } L^2(0,T;\Hs) \label{eq:schauder_betal_convergence}
    \end{alignat}
    along a subsequence of $\lambda$ that we do not relabel. The passage to the limit in the approximate system, can be performed exactly as in the proof of \Cref{prop:existence_lambda}, proving strong convergence where needed through standard compactness result and directly for the displacement equation. The only difference that needs further discussion is the Yosida approximation in the damage equation because we need to justify that $\xi \in \beta(z)$. From \eqref{eq:schauder_zl_estimate} and Aubin--Lions compact embedding
    \begin{equation*}
           H^1(0,T;\Hs) \cap L^2(0,T;\Ws) \hookrightarrow \hookrightarrow L^2(0,T;\Vs),
    \end{equation*}
    $z_{\lambda} \to z$ strongly in $L^2(0,T;\Hs)$ along a further non-relabeled subsequence. Thus, since $\beta$ is maximal monotone and $\beta_{\lambda}$ is its Yosida approximation, we exploit \cite[Proposition 1.1, p. 42]{barbu2010} and deduce that $\xi \in \beta(z)$ a.e. in $Q$. This concludes the proof of \Cref{thm:existence}.\\

\subsection{Proof of \Cref{thm:regularity}}
\label{sec:th2}

\noindent As we will comment further later, the prescribed regularity can be proved at the approximate level using standard regularity results. Then, the only thing to be shown is that such regularity passes to the limit and this will be done by proving some a priori estimates in the stronger norms we need.

\begin{lemma}
     Under the Hypothesis \textbf{(A)}-\textbf{(B)}, the solution to the approximate problem found in \Cref{prop:existence_lambda} enjoys the further regularity
    \begin{gather*}
        \varphi_{\lambda} \in H^1(0,T;\Vs) \cap L^{\infty}(0,T;\Ws) \cap L^2(0,T;H^3),\\
         \uu_{\lambda} \in W^{1,\infty}(0,T;\Wsz),\\
         z_{\lambda} \in  W^{1,\infty}(0,T;\Hs) \cap H^1(0,T;\Vs)  \cap L^{\infty}(0,T;\Ws).
    \end{gather*}
     Moreover, it satisfies
    \begin{align}
         \norm{\varphi_{\lambda}}_{H^1(\Vs)\cap L^{\infty}(\Ws) \cap L^2(H^3)} &\leq C,\label{eq:schauder_phil_reg_estimate}\\
         \norm{\uu_{\lambda}}_{W^{1,\infty}(\Wsz)}&\leq C,\label{eq:schauder_ul_reg_estimate}\\
         \norm{z_{\lambda}}_{H^1(\Vs) \cap L^{\infty}(\Ws)}&\leq C,\label{eq:schauder_zl_reg_estimate}\\
          \norm{\beta_{\lambda}(z_{\lambda})}_{ L^{\infty}(\Hs)}&\leq C,\label{eq:schauder_betal_reg_estimate}
    \end{align}
    for a positive constant $C$ depending only on the problem's data and not on $\lambda$.
\end{lemma}

\begin{proof}
    To simplify the notation, since here $\lambda$ is fixed, we will omit it.\\
    
    \noindent \textit{Estimate for }$\varphi$. The right-hand side of equation \eqref{eq:phi_lambda} is
    \begin{equation*}
        h \coloneq p(\sigma,z)\varphi \left(1-\frac{\varphi}{N}\right)  - \varphi g(\sigma,z).
    \end{equation*}
    Notice that we removed the truncation $\alpha$ because we have already proved that $0 \leq \varphi \leq N$ (cf. \eqref{eq:aprox_syst_phi_sigma_boundedness} in \Cref{prop:existence_lambda}). From well-known regularity theory for parabolic equations (see \cite{Dautray_Lions_92,Lions61}), since  $h \in L^2(\Vs)$, $\varphi$ enjoys the regularity declared in the statement and it holds 
    \begin{equation} \label{eq:estimate_reg_phi_partial}
        \norm{\varphi}_{H^1(\Vs)\cap L^{\infty}(\Ws) \cap L^2(H^3)} \leq C (\norm{h}_{L^2(\Vs)} + \norm{\varphi_0}_{\Ws}).
    \end{equation}
    We aim to bound $\norm{h}_{L^2(\Vs)}$. Since $p$ and $g$ are bounded by Hypothesis \ref{hyp:eq_phi}, it follows that
    \begin{equation*}
        \norm{h}_{L^2(\Hs)} \leq (p^* + g^*) N |\Omega|^{\frac{1}{2}}.
    \end{equation*}
    Moreover, we have that
    \begin{multline*}
        \nabla h =  (p_{,\sigma}\nabla\sigma + p_{,z}\nabla z) \varphi \left(1-\frac{\varphi}{N}\right) + p(\sigma,z)\left[\nabla \varphi \left(1-\frac{\varphi}{N}\right)-\varphi\frac{\nabla \varphi}{N}\right]\\
            - g(\sigma,z) \nabla \varphi - (g_{,\sigma}\nabla\sigma + g_{,z}\nabla z)\varphi.
    \end{multline*}
    Recalling that $\varphi$ and $z$ are bounded, $p$ and $g$ are bounded along with their partial derivatives by Hypotheses \ref{hyp:eq_phi} and \ref{hyp:uniqueness_eq_phi}, we have
    \begin{equation*}
        |\nabla h| \leq C (|\nabla \sigma| + |\nabla z| + |\nabla \varphi|). 
    \end{equation*}
    Thus, since we have already proved \eqref{eq:schauder_sigmal_estimate}, \eqref{eq:schauder_zl_estimate}, \eqref{eq:schauder_phil_estimate}, we get that $\norm{\nabla h}_{L^2(\Hs)}$ is uniformly bounded. Hence, from \eqref{eq:estimate_reg_phi_partial}, we obtain \eqref{eq:schauder_phil_reg_estimate}.\\

    \noindent \textit{Estimate for }$z$ \textit{and} $\beta_{\lambda}(z)$. Observing that the term
    \begin{equation*}
         w - \Psi(\varphi,\epsilonu) - \beta_{\lambda}(z) - \pi(z) 
    \end{equation*}
    belongs to $H^1(0,T;\Hs)$, by standard parabolic regularity results (see \cite{Dautray_Lions_92,Lions61}), $z$ has the desired regularity. Unfortunately, the estimate \eqref{eq:schauder_zl_reg_estimate} (which is independent of $\lambda$) can not be deduced as we have done for $\varphi$ because the Lipschitz constant of the Yosida approximation $\beta_{\lambda}$ depends on $\lambda$. To overcome this difficulty, we aim to test equation \eqref{eq:damage_lambda} by
    $\partial_t(-\Delta z + \beta_{\lambda}(z))$ and integrate over the time interval $(0,t)$. Notice that the following calculations are formal, since
    $\partial_t(-\Delta z + \beta_{\lambda}(z))$ does not possess the regularity $L^2(0,T;\Hs)$, which would be suitable for \Cref{eq:damage_lambda}. However, the same estimate can be performed rigorously at the discrete level in a Galerkin scheme, so we will not enter into technical details. We have:
    \begin{align*}
        \intto |\nabla&(\partial_t z)|^2 \dx\dtau + \intto \beta_{\lambda}'(z)|\partial_t z|^2 \dx\dtau\\
        &\quad + \frac{1}{2} \left[ \into |-\Delta z + \beta_{\lambda}(z)|^2 \dx - \into |-\Delta z_0 + \beta_{\lambda}(z_0)|^2 \dx \right]\\ 
        &= -\intto (\pi(z)+ \Psi(\varphi,\epsilonu) - w)\partial_t(-\Delta z + \beta_{\lambda}(z))\dx\dtau.
    \end{align*}
    Then we integrate by parts the term on the right-hand side with respect to time, obtaining:
    \begin{equation}\label{eq:estimate_reg_beta_partial}
        \begin{split}
            \intt&\into |\nabla(\partial_t z)|^2 \dx\dtau + \intto \beta_{\lambda}'(z)|\partial_t z|^2 \dx\dtau + \frac{1}{2} \into |-\Delta z + \beta_{\lambda}(z)|^2 \dx\\ 
            =&\into (\pi(z_0)+ \Psi(\varphi_0,\varepsilon(\uu_0)) - w(0))(-\Delta z_0 + \beta_{\lambda}(z_0)) \dx\\
            &- \into (\pi(z)+ \Psi(\varphi,\epsilonu) - w)(-\Delta z + \beta_{\lambda}(z)) \dx\\
            &+ \intto  \left(\pi'(z)\partial_t z + \partial_t \Psi(\varphi,\epsilonu) - \partial_t w\right)(-\Delta z + \beta_{\lambda}(z)) \dx\dtau\\
            &+ \frac{1}{2} \into |-\Delta z_0 + \beta_{\lambda}(z_0)|^2 \dx\\
            \eqcolon& \, I_1 + I_2 + I_3 + I_4.
        \end{split}
    \end{equation}
    Regarding the left-hand side, since $\beta_{\lambda}'$ is non negative because $\beta_{\lambda}$ is monotone, we have
    \begin{equation}\label{eq:estimate_reg_beta_lh}
        \begin{split}
            &\intto |\nabla(\partial_t z)|^2 \dx\dtau + \frac{1}{2} \into |-\Delta z + \beta_{\lambda}(z)|^2 \dx\\
            &\leq \intto |\nabla(\partial_t z)|^2 \dx\dtau + \intto \beta_{\lambda}'(z)|\partial_t z|^2 \dx\dtau + \frac{1}{2} \into |-\Delta z + \beta_{\lambda}(z)|^2 \dx.
        \end{split}
    \end{equation}
    Next, we aim to bind the terms on the right-hand side. As regard $I_1$ and $I_4$, which depend on the initial data, we only observe that, by well-known properties of the Yosida approximation (see e.g. \cite{brezis1973}),
    \begin{equation*}
        |\beta_{\lambda}(z_0)| \leq |\beta^0(z_0)|,
    \end{equation*}
    where the right-hand side is bounded because of Hypothesis \ref{hyp:uniqueness_initial_conditions}. All the other addends in $I_1$ and $I_4$ do not depend on $\lambda$ and are bounded because of the Hypotheses we imposed on the initial data.
    Concerning $I_2$, we recall that $\pi$ and $\Psi$ are Lipschitz continuous by Hypotheses \ref{hyp:pi} and \ref{hyp:psi}. Then we apply Young inequality. We get
    \begin{equation}\label{eq:estimate_reg_beta_rh1}
        \begin{split}
            I_2 &\leq C \into \left(|z|+ 1+ |\epsilonu| + |\varphi| + |w|\right)|-\Delta z + \beta_{\lambda}(z)| \dx\\
            &\leq \frac{1}{4} \into |-\Delta z + \beta_{\lambda}(z)|^2 \dx + C \into \left( |z|^2 + 1 + |\epsilonu|^2 + |\varphi|^2 + |w|^2 \right)\dx\\
            &\leq \frac{1}{4} \into |-\Delta z + \beta_{\lambda}(z)|^2 \dx + C,
        \end{split}
    \end{equation}
    where the last inequality is due to Hypothesis \ref{hyp:w} and to the fact that $z$, $\epsilonu$, and $\varphi$ are uniformly bounded in $L^{\infty}(0,T;\Hs)$ by \eqref{eq:schauder_phil_estimate}, \eqref{eq:schauder_ul_estimate} and \eqref{eq:schauder_zl_estimate}. The term $I_3$ can be handled similarly, by using in particular \ref{hyp:psi}. We have
    \begin{equation}\label{eq:estimate_reg_beta_rh2}
        \begin{split}
            I_3 &\leq C \intto  \left(|\partial_t z| + |\epsilonut| + |\partial_t\varphi| + |\partial_t w|\right)|-\Delta z + \beta_{\lambda}(z)| \dx\dtau\\
            & \leq C \intto  |-\Delta z + \beta_{\lambda}(z)|^2 \dx\dtau\\
            &\quad + C\intto \left(|\partial_t z|^2 + |\epsilonut|^2 + |\partial_t\varphi|^2 + |\partial_t w|^2\right) \dx \dtau\\
            &\leq C \intto  |-\Delta z + \beta_{\lambda}(z)|^2 \dx\dtau + C,
        \end{split}
    \end{equation}
    recalling Hypothesis \ref{hyp:uniqueness_w} and the fact that $\partial_t z$, $\epsilonut$ and $\partial_t \varphi$ are uniformly bounded in $L^{2}(0,T;\Hs)$ again by \eqref{eq:schauder_phil_estimate}, \eqref{eq:schauder_ul_estimate}, \eqref{eq:schauder_zl_estimate}. Putting together \eqref{eq:estimate_reg_beta_partial}, \eqref{eq:estimate_reg_beta_lh}, \eqref{eq:estimate_reg_beta_rh1}, \eqref{eq:estimate_reg_beta_rh2} yields to
    \begin{multline*}
        \intto |\nabla(\partial_t z)|^2 \dx\dtau + \frac{1}{4}\into |-\Delta z + \beta_{\lambda}(z)|^2 \dx \leq C + C \intto |-\Delta z + \beta_{\lambda}(z)|^2 \dx \dtau.
    \end{multline*}
    Thus, applying Gronwall Lemma, we conclude that the left-hand side is uniformly bounded. Moreover, since $\beta_{\lambda}'$ is non-negative, it holds 
    \begin{equation*}
        \begin{split}
            C &\geq \into |-\Delta z + \beta_{\lambda}(z)|^2 \dx =  \into |-\Delta z|^2 + |\beta_{\lambda}(z)|^2 -2\Delta z \beta_{\lambda}(z) \dx \\
            &= \into |-\Delta z|^2 + |\beta_{\lambda}(z)|^2 +2\beta_{\lambda}'(z) |\nabla z|^2 \dx \geq \into |-\Delta z|^2 + |\beta_{\lambda}(z)|^2 \dx. 
        \end{split}
    \end{equation*}
    and, consequentially, estimates \eqref{eq:schauder_betal_reg_estimate} and \eqref{eq:schauder_zl_reg_estimate} are satisfied.\\

    \noindent \textit{Estimate for }$\uu$. The additional regularity for $\uu$ can be proved at the time-discrete level, noticing that, starting from $\uu_0 \in \Wsz$, the right-hand side in equation \eqref{eq:discrete_problem_u_rewritten} can be represented as a linear functional on $\Hs$. As a consequence, applying the regularity result \cite[][Theorem 1.11, p. 322]{marsden1994}, the time-discrete solution is in $\Ws$. To prove that the limit solution preserves this regularity, one should prove an estimate uniform with respect to the time-step $\tau$. Anyway, we are going to omit these calculations since they are similar to the ones we are going to perform in the continuous setting, where we search for an estimate uniform in $\lambda$.\\
    Since $\uu \in W^{1,\infty}(0,T;\Wsz)$, equation \eqref{eq:displacement_lambda} can be rewritten as
    \begin{multline*}
        -\A(\varphi,z) \diver[\epsilonut] = \B(\varphi,z)\diver[\epsilonu] 
        + \epsilonut \left[ \A_{,\varphi}(\varphi,z)\nabla\varphi + \A_{,z}(\varphi,z)\nabla z\right] \\
        + \epsilonu \left[ \B_{,\varphi}(\varphi,z)\nabla\varphi + \B_{,z}(\varphi,z)\nabla z\right] + \ff
    \end{multline*}
    which is satisfied in $\Hs$ for a.e. $t \in (0,T)$. Multiplying it by $-\diver[\epsilonut]$ and integrating over $\Omega$, one gets
    \begin{equation}\label{eq:estimate_reg_u_partial}
        \begin{split}
            \into &\A(\varphi,z) \diver[\epsilonut] \cdot \diver[\epsilonut]\dx =  -\into \B(\varphi,z)\diver[\epsilonu] \cdot \diver[\epsilonut]\dx\\
            &-  \into \epsilonut \left[ \A_{,\varphi}(\varphi,z)\nabla\varphi - \A_{,z}(\varphi,z)\nabla z\right] \cdot \diver[\epsilonut]\dx \\
            &-  \into\epsilonu \left[ \B_{,\varphi}(\varphi,z)\nabla\varphi - \B_{,z}(\varphi,z)\nabla z\right] \cdot \diver[\epsilonut]\dx \\
            &- \into \ff\cdot \diver[\epsilonut]\dx
            \eqcolon I_1 + I_2 + I_3 +I_4.
        \end{split}
    \end{equation}
    Recalling that, by Hypothesis \ref{hyp:eq_displacement}, $\A$ is strictly positive definite, 
    \begin{equation}\label{eq:estimate_reg_u_coercivity}
         C_{{\A}_*}\norm{\diver[\epsilonut]}_{\Hs}^2 \leq \into \A(\varphi,z) \diver[\epsilonut] \cdot \diver[\epsilonut]\dx.
    \end{equation}
    Now we are going to estimate the terms on the right-hand side of equation \eqref{eq:estimate_reg_u_partial}. Starting from $I_1$, we employ the boundedness of $\B$ from Hypothesis \ref{hyp:eq_displacement}, Young inequality with a small parameter $\eta$ yet to be determined, and the following equality
    \begin{equation*}
        \diver[\epsilonu](\tau) = \diver[\varepsilon(\uu_0)] + \inttau \diver[\epsilonut] \dt 
    \end{equation*}
    and we obtain
    \begin{equation} \label{eq:estimate_reg_u_I1}
        \begin{split}
            I_1 &\leq \eta \norm{\diver[\epsilonut]}_{\Hs}^2 + C_{\eta} \norm{\diver[\epsilonu]}_{\Hs}^2\\
            & \leq \eta \norm{\diver[\epsilonut]}_{\Hs}^2 + C_{\eta} \left(\norm{\diver[\varepsilon(\uu_0)]}_{\Hs}^2 + \inttau \norm{\diver[\epsilonut]}_{\Hs}^2\dt\right),
        \end{split}
    \end{equation}
    where the constant $C_{\eta}$ has changed from the first in the passage from the first to the second line.\\
    Regarding $I_2$, by Lipschitz continuity of $\A$ and $\B$, H\"older inequality and Young inequality with a small $\eta$, we get
    \begin{equation}\label{eq:estimate_reg_u_I2_partial}
        \begin{split}
            I_2 \leq \eta \norm{\diver[\epsilonut]}_{\Hs}^2  + C_{\eta} \norm{\epsilonut}_{L^4}^2 \left( \norm{\nabla \varphi}_{L^4}^2+\norm{\nabla z}_{L^4}^2\right).
        \end{split}
    \end{equation}
    Since we have already shown that $\norm{\varphi}_{\Ws}, \norm{z}_{\Ws} \leq C$ from \eqref{eq:schauder_phil_reg_estimate}, \eqref{eq:schauder_zl_reg_estimate}, by the continuous embedding $\Ws \hookrightarrow W^{1,4}$ it follows 
    \begin{equation}\label{eq:estimate_reg_u_I2_phiz}
        \norm{\nabla \varphi}_{L^4}^2+\norm{\nabla z}_{L^4}^2 \leq C.
    \end{equation}
    Employing the Ehrling's Lemma with $\Ws\hookrightarrow\hookrightarrow W^{1,4} \hookrightarrow \Hs$ and a small constant $\theta$ yet to be determined, it holds 
    \begin{equation*}
        \begin{split}
            \norm{\epsilonut}_{L^4}^2 &\leq C  \norm{\uut}_{W^{1,4}}^2 \leq  \theta \norm{\uut}_{\Ws}^2 + C_{\theta}\norm{\uut}_{\Hs}^2\\
            &\leq \theta \norm{\diver[\epsilonut]}_{\Hs}^2 + C_{\theta}. 
        \end{split}
    \end{equation*}
    The last inequality follows from \Cref{lemma:H2_inequality} and the previous estimate \eqref{eq:schauder_ul_estimate}, renaming the constant involved. Putting these elements together, starting from \eqref{eq:estimate_reg_u_I2_partial} we have proved that
    \begin{equation}\label{eq:estimate_reg_u_I2}
        I_2 \leq (\eta + \theta)\norm{\diver[\epsilonut]}_{\Hs}^2 + C_{\eta,\theta}.
    \end{equation}
    The term $I_3$ can be treated similarly. Proceeding as in equation \eqref{eq:estimate_reg_u_I2_partial} and taking into account equation \eqref{eq:estimate_reg_u_I2_phiz}, yields to
    \begin{equation*}
        I_3 \leq \eta \norm{\diver[\epsilonut]}_{\Hs}^2  + C_{\eta} \norm{\epsilonu}_{L^4}^2 \leq \eta \norm{\diver[\epsilonut]}_{\Hs}^2  + C_{\eta} \norm{\uu}_{\Ws}^2.
    \end{equation*}
    Since 
    \begin{equation*}
        \uu(\tau) = \uu_0 + \inttau \uut \dt,
    \end{equation*}
    we have
    \begin{equation}\label{eq:estimate_reg_u_I3}
         \begin{split}
             I_3 &\leq \eta \norm{\diver[\epsilonut]}_{\Hs}^2  + C_{\eta}\left( \norm{\uu_0}_{\Ws}^2 + \inttau \norm{\uut}_{\Ws}^2 \dt \right)\\
             & \leq \eta \norm{\diver[\epsilonut]}_{\Hs}^2  + C_{\eta}\left( 1 + \inttau \norm{\diver[\epsilonut]}_{\Hs}^2  \dt \right),
         \end{split}
    \end{equation}
    where the last inequality follows from \Cref{lemma:H2_inequality}, changing the constant $C_{\eta}$. Finally, by Young and H\"older inequalities,
    \begin{equation}\label{eq:estimate_reg_u_I4}
        I_4 \leq \norm{\ff}_{\Hs} \norm{\diver[\epsilonut]}_{\Hs} \leq C_{\eta} \norm{\ff}_{\Hs}^2 + \eta \norm{\diver[\epsilonut]}_{\Hs}^2.
    \end{equation}
    Now we put together \eqref{eq:estimate_reg_u_partial} with \eqref{eq:estimate_reg_u_coercivity}, \eqref{eq:estimate_reg_u_I1}, \eqref{eq:estimate_reg_u_I2}, \eqref{eq:estimate_reg_u_I3}, \eqref{eq:estimate_reg_u_I4} and we move to the left-hand side the terms with $\eta$ and $\theta$, fixing them small enough. We obtain:
    \begin{equation}
        \begin{split}
             \norm{\diver[\epsilonut]}_{\Hs}^2(\tau) \leq C \left( 1 + \norm{\ff}_{L^{\infty}(\Hs)}^2 + \inttau \norm{\diver[\epsilonut]}_{\Hs}^2 \dt \right).
        \end{split}
    \end{equation}
    Applying Gronwall's inequality, we prove that $\norm{\uut}_{L^{\infty}(\Ws)}$ is uniformly bounded and, as a consequence, that \eqref{eq:schauder_ul_reg_estimate} holds. This concludes the proof of \Cref{thm:regularity}.
\end{proof}

\subsection{Proof of \Cref{thm:continous_dependence}}
\label{sec:th3}

    \noindent  Consider two pairs $\{(\varphi_i,\sigma_i,\uu_i,z_i)\}_{i=1,2}$ of weak solution corresponding to the assigned functions $\{(\ff_i, w_i, \sigma_{\Gamma,i})\}_{i=1,2}$ and to the initial data $\{(\varphi_{0,i},\sigma_{0,i},\uu_{0,i},z_{0,i})\}_{i=1,2}$. For the sake of brevity, in the following, we will use the shorter notation
    \begin{gather*}
        \varphi \coloneqq \varphi_{1} - \varphi_2, \qquad  \sigma \coloneqq \sigma_{1} - \sigma_2,  \qquad 
        \uu\coloneqq \uu_{1} - \uu_2, \qquad
        z \coloneqq z_{1} - z_2,\\
        \varphi_0 \coloneqq \varphi_{0,1} - \varphi_{0,2}, \qquad  \sigma_0 \coloneqq \sigma_{0,1} - \sigma_{0,2},  \qquad 
        \uu_0\coloneqq \uu_{0,1} - \uu_{0,2}, \qquad
        z_0 \coloneqq z_{0,1} - z_{0,2},\\
        \ff \coloneqq \ff_1 -\ff_2, \qquad w \coloneqq w_1 - w_2, \qquad \sigma_{\Gamma} \coloneqq \sigma_{\Gamma,1} - \sigma_{\Gamma,2},
    \end{gather*}
    and we will denote $M \coloneqq \max_{i=1,2} M_i$.
    Taking the difference of \eqref{eq:phi_with_spaces} written for $\varphi_1$ and $\varphi_2$, testing it by $\varphi$, we obtain
    \begin{equation}\label{eq:test_uniqueness_phi}
        \begin{split}
        \frac{1}{2}&\frac{\text{d}}{\dt}\into |\varphi|^2 \dx + \into |\nabla \varphi|^2 \dx=
        - \into \left(g(\sigma_1,z_1)-g(\sigma_2,z_2)\right)\varphi_1\varphi \dx\\
        &- \into g(\sigma_2,z_2) \varphi^2 \dx + \into \left(p(\sigma_1,z_1)-p(\sigma_2,z_2)\right)\varphi_1\Big(1-\frac{\varphi_1}{N}\Big)\varphi \dx\\
        &+\into p(\sigma_2,z_2)\bigg[\varphi_1\Big(1-\frac{\varphi_1}{N}\Big)-\varphi_2\Big(1-\frac{\varphi_2}{N}\Big)\bigg]\varphi \dx\\
        \leq& C \bigg[\into (|\sigma|+|z|)\varphi \dx + \into |\varphi|^2 \dx \bigg] \leq C \bigg[\into \left(|\sigma|^2+|z|^2 + |\varphi|^2 \right) \dx \bigg]
        \end{split}
    \end{equation}
    where the inequality follows from Young inequality and the constant $C$ depends on $N$, $g^*$, $p^*$, and the  Lipschitz constants of $p$ and $g$. 
    Taking the difference of \eqref{eq:sigma_with_spaces} written for $\sigma_1$ and $\sigma_2$, testing it by $\sigma$, we obtain
    \begin{align*}
        \frac{1}{2}&\frac{\text{d}}{\dt}\into |\sigma|^2 \dx + \into |\nabla \sigma|^2 \dx + \intg |\sigma|^2 \dS + \into \frac{k_1(\varphi_1,z_1)}{k_2(\varphi_1,z_1)+\sigma_1}\sigma^2 \dx \\
        =&\intg \sigma_{\Gamma} \sigma \dS -\into \frac{k_1(\varphi_1,z_1)-k_1(\varphi_2,z_2)}{k_2(\varphi_1,z_1)+\sigma_1}\sigma_2\sigma \dx\\
        &+\into \bigg(\frac{k_1(\varphi_2,z_2)}{k_2(\varphi_1,z_1)+\sigma_1}-\frac{k_1(\varphi_2,z_2)}{k_2(\varphi_2,z_2)+\sigma_2}\bigg) \sigma_2\sigma \dx 
        + \into (J(\varphi_1,z_1)-J(\varphi_2,z_2)) \sigma \dx\\
        \leq& \, \frac{1}{2}\intg |\sigma_{\Gamma}|^2 \dS +\frac{1}{2} \intg |\sigma|^2 \dS + C \into (|\varphi|+|z| + |\sigma|)|\sigma| \dx
    \end{align*}
   where we have used Hypotheses \ref{hyp:eq_lactate} and \ref{hyp:uniqueness_eq_lactate} for $k_1$, $k_2$ and $J$, the fact that $0 \leq \sigma_i \leq M$ and Young inequality. Here the constant $C$ depends on $M$, $k_1^*$, ${k_2}_*$, $J^*$ and the Lipschitz constant of $k_1$, $k_2$, $J$. It follows that
   \begin{equation}\label{eq:test_uniqueness_sigma}
       \begin{split}
           \frac{1}{2}\frac{\text{d}}{\dt}\into |\sigma|^2 \dx + \into |\nabla \sigma|^2 \dx &+ \frac{1}{2}\intg |\sigma|^2 \dS\\
            &\leq \frac{1}{2}\intg |\sigma_{\Gamma}|^2 \dS  + C \into \left(|\varphi|^2+|z|^2 + |\sigma|^2 \right) \dx.
       \end{split}
   \end{equation}
   We take the difference of \eqref{eq:u_with_spaces} written for $\uu_1$ and $\uu_2$ and test it by $\partial_t\uu$. Integrating over $\Omega$, summing and subtracting some terms and using $\A$ positive definiteness, we get
   \begin{equation}\label{eq:test_uniqueness_u_incomplete}
       \begin{split}
           &{C_{\A}}_*\!\into |\epsilonut|^2 \dx\leq \into \A(\varphi_1,z_1)\varepsilon(\uut):\epsilonut \dx \\
           &= -\!\into \!(\A(\varphi_1,z_1) - \A(\varphi_2,z_2)) \varepsilon(\uut_2):\epsilonut \dx
           -  \!\into \!\B(\varphi_1,z_1)\varepsilon(\uu):\epsilonut \dx\\
           &\quad- \into (\B(\varphi_1,z_1) - \B(\varphi_2,z_2)) \varepsilon(\uu_2):\epsilonut \dx
           + \into \ff \cdot \uut \dx\\
           &\eqcolon \, I_1 + I_2 + I_3 + I_4.
       \end{split}
   \end{equation}
   Our next aim is to provide a bound for each integral on the right-hand side. Regarding $I_1$, we employ Lipschitz continuity of $\A$ and H\"older inequality. We obtain
   \begin{equation*}
       \begin{split}
           I_1 &
           \leq C \into (|\varphi|+|z|)|\varepsilon(\uut_2)||\epsilonut| \dx 
           \leq C (\norm{\varphi}_{L^3} + \norm{z}_{L^3})\norm{\varepsilon(\uut_2)}_{L^6}\norm{\epsilonut}_{\Hs}\\
           &\leq \eta \norm{\epsilonut}_{\Hs}^2 + C_{\eta} (\norm{\varphi}_{L^3}^2 + \norm{z}_{L^3}^2),
       \end{split}
   \end{equation*}
   where the last inequality holds because $\varepsilon(\uut_2)$ is uniformly bounded in $L^{\infty}(\Vs) \hookrightarrow L^{\infty}(L^6)$ and we have applied Young inequality with a small constant $\eta$. Notice that $C_{\eta}$ depends on $\max_{i=1,2}\left(\norm{\varepsilon(\partial_t\uu_i)}_{L^{\infty}(\Vs)}\right)$. Employing \Cref{lemma:special_case_gagliardo_nirenberg} and then again Young inequality yields to
   \begin{equation}\label{eq:ineq_uniqueness_I1}
       \begin{split}
           I_1  &\leq \eta \norm{\epsilonut}_{\Hs}^2 + C_{\eta} (\norm{\varphi}_{\Hs}^{1/2}\norm{\varphi}_{\Vs}^{1/2} + \norm{z}_{\Hs}^{1/2}\norm{z}_{\Vs}^{1/2})^2\\
           &\leq  \eta (\norm{\epsilonut}_{\Hs}^2 + \norm{\varphi}_{\Vs}^2 + \norm{z}_{\Vs}^2)+  C_{\eta}(\norm{\varphi}_{\Hs}^2 + \norm{z}_{\Hs}^2)\\
           &\leq \eta \bigg( \into |\epsilonut|^2 \dx + \into |\nabla \varphi|^2 \dx + \into |\nabla z|^2 \dx  \bigg)\\
           &\quad+ C_{\eta} \bigg(  \into |\varphi|^2 \dx + \into | z|^2 \dx  \bigg).
       \end{split}
   \end{equation}
   Concerning $I_2$, using H\"older and Young inequalities, we get that
   \begin{equation}\label{eq:ineq_uniqueness_I2}
       \begin{split}
           I_2 &\leq C \into |\epsilonu| |\epsilonut| \dx \leq \eta \into |\epsilonut|^2 \dx + C_{\eta} \into |\epsilonu|^2 \dx\\  
           &\leq \eta \into |\epsilonut|^2 \dx + C_{\eta} \bigg[ \into |\varepsilon(\uu_0)|^2 \dx + \intto |\epsilonut|^2 \dx \dtau\bigg]. 
       \end{split}
   \end{equation}
   The integral $I_3$ can be treated exactly as $I_1$ and satisfies the same estimate
   \begin{equation}\label{eq:ineq_uniqueness_I3}
       \begin{split}
           I_3  \leq  \eta \bigg( \into |\epsilonut|^2 \dx + \into |\nabla \varphi|^2 \dx + \into |\nabla z|^2 \dx  \bigg)\\
           + C_{\eta} \bigg(  \into |\varphi|^2 \dx + \into | z|^2 \dx  \bigg),
       \end{split}
   \end{equation}
   where $C_{\eta}$ depends on $\max_{i=1,2}\left(\norm{\varepsilon(\uu_i)}_{L^{\infty}(\Vs)}\right)$.
    Finally, $I_4$ can be handled by employing H\"older and Young inequalities again, in order to obtain
   \begin{equation}\label{eq:ineq_uniqueness_I4}
       \begin{split}
           I_4 & \leq \eta \into |\epsilonut|^2 \dx +  C_{\eta} \into |\ff|^2 \dx. 
       \end{split}
   \end{equation}
   Taking advantage of the inequalities \eqref{eq:ineq_uniqueness_I1}, \eqref{eq:ineq_uniqueness_I2}, \eqref{eq:ineq_uniqueness_I3}, \eqref{eq:ineq_uniqueness_I4} in \eqref{eq:test_uniqueness_u_incomplete} and redefining $\eta$, we obtain:
   \begin{align}\label{eq:test_uniqueness_u}
       \begin{split}
           {C_{\A}}_*&\into |\epsilonut|^2 \dx\\
           &\leq \eta \bigg( \into |\epsilonut|^2 \dx + \into |\nabla \varphi|^2 \dx + \into |\nabla z|^2 \dx  \bigg)
            + C_{\eta} \bigg(  \into |\varphi|^2 \dx\\
            & \quad + \into | z|^2 \dx + \into |\varepsilon(\uu_0)|^2 \dx + \intto |\epsilonut|^2 \dx \dt + \into |\ff|^2 \dx\bigg).
       \end{split}
   \end{align}
   Finally, we take the difference between \eqref{eq:z_with_spaces} written for $z_1$ and $z_2$ and test it by $z$. Integrating over $\Omega$, exploiting  monotonicity of $\beta$, Lipschitz continuity of $\pi$ and $\Psi$ and Young inequality, we deduce that
   \begin{equation}\label{eq:test_uniqueness_z}
       \begin{split}
           \frac{1}{2}&\frac{\text{d}}{\dt} \!\into |z|^2 \dx + \!\into |\nabla z|^2 \dx \leq \frac{1}{2}\frac{\text{d}}{\dt} \! \into |z|^2 \dx + \! \into |\nabla z|^2 \dx + \!\into (\xi_1-\xi_2)z \dx \\
           &= -\into (\pi(z_1)-\pi(z_2))z \dx + \into w \, z \dx\\
           &\quad - \into \left[\Psi(\varphi_1,\varepsilon(\uu_1))-\Psi(\varphi_2,\varepsilon(\uu_2))\right]z \dx\\
           &\leq C\bigg(\into |z|^2 \dx + \into |w|^2 \dx + \into |\varphi|^2 \dx + \into |\epsilonu|^2 \dx \bigg),
       \end{split}
   \end{equation}
   where $C$ depends only on the Lipschitz constants of $\pi$ and $\Psi$. At this point, summing the inequalities \eqref{eq:test_uniqueness_phi}, \eqref{eq:test_uniqueness_sigma}, \eqref{eq:test_uniqueness_u}, \eqref{eq:test_uniqueness_z} and moving the terms multiplied by $\eta$ to the left-hand side, we get
   \begin{equation}
       \begin{split}
            &\frac{\text{d}}{\dt}\bigg(\into |\varphi|^2 \dx + \into |\sigma|^2 \dx + \into |z|^2 \dx \bigg)\\
           &\quad+ \into |\nabla \varphi|^2 \dx + \into |\nabla \sigma|^2 \dx + \into |\nabla z|^2 \dx +  \into |\epsilonut|^2 \dx    + \intg |\sigma|^2 \dS\\
           &\leq C \bigg[\into |\sigma|^2 \dx + \into |z|^2 \dx  + \into |\varphi|^2 \dx +  \intto |\epsilonut|^2 \dx \dtau \bigg] \\
           &\quad + C \bigg[\into |\varepsilon(\uu_0)|^2 \dx + \intg |\sigmag|^2 \dS +  \into |w|^2 \dx + \into |\ff|^2 \dx \bigg].
       \end{split}
    \end{equation}
   Integrating in time and then applying Gronwall inequality, the following estimate follows
   \begin{multline}
       \norm{\varphi}_{L^{\infty}(\Hs)\cap L^2(\Vs)} + \norm{\sigma}_{L^{\infty}(\Hs)\cap L^2(\Vs)} + \norm{z}_{L^{\infty}(\Hs)\cap L^2(\Vs)} + \norm{\uu}_{H^1(\Vs_0)} \\
       \leq C \Big( \norm{\varphi_0}_{\Hs} 
        + \norm{\sigma_0}_{\Hs} + \norm{z_0}_{\Hs} + \norm{\uu_0}_{\Vs_0}\\
         +  \norm{\ff}_{L^2(\Hs)} + \norm{w}_{L^2(\Hs)} +  \norm{\sigmag}_{L^2(L^2_{\Gamma})}  \Big)
   \end{multline}
   for a constant $C$ that does not depend on the differences $\varphi$, $\sigma$, $\uu$ and $z$ but depends on the fixed data of the problem, $T$ and  $\max_{i=1,2}(\norm{\varepsilon(\partial_t\uu_i)}_{L^{\infty}(\Vs)})$. This concludes the proof of \Cref{thm:continous_dependence}.

\section*{Acknowledgments}
G.C. wishes to acknowledge the financial support of the International Research Laboratory LYSM, IRL 2019 CNRS/INdAM. P.C. and E.R. gratefully mention some support from the Next Generation EU Project 
No.P2022Z7ZAJ (A unitary mathematical framework for modelling 
muscular dystrophies), and the MIUR-PRIN Grant
2020F3NCPX “Mathematics for industry 4.0
(Math4I4)". 
G.C., P.C. and E.R. are members of GNAMPA (Gruppo Nazionale per l’Analisi Matematica, la Probabilità e le loro Applicazioni) of INdAM (Istituto Nazionale di Alta Matematica).

\printbibliography

\end{document}